\documentclass[12pt]{scrartcl}

\usepackage{amsmath,amsfonts,amsthm,amssymb,graphicx,float,tikz,indentfirst,graphics,setspace,young,mathrsfs,bbm}
\usepackage[hidelinks]{hyperref}
\usepackage[enableskew,vcentermath]{youngtab}
\usepackage[all]{xy}
\usetikzlibrary{decorations.pathreplacing,patterns}
\theoremstyle{plain}
\newtheorem{theorem}{Theorem}[section]
\newtheorem{conjecture}[theorem]{Conjecture}
\newtheorem{corollary}[theorem]{Corollary}
\newtheorem{lemma}[theorem]{Lemma}
\newtheorem{proposition}[theorem]{Proposition}

\theoremstyle{definition}
\newtheorem{definition}[theorem]{Definition}
\newtheorem{algorithm}[theorem]{Algorithm}
\newtheorem{example}[theorem]{Example}

\newtheorem{remark}[theorem]{Remark}

\newcommand{\rr}{\mathbb{R}}

\newcommand{\calD}{\mathcal{D}}

\newcommand{\calP}{\mathcal{P}}

\newcommand{\calR}{\mathcal{R}}
\newcommand{\calS}{\mathcal{S}}

\newcommand{\Le}{\reflectbox{L}}

\title{Combinatorics of the Deodhar decomposition of the Grassmannian}
\author{Cameron Marcott}
\date{}

\begin{document}
\maketitle

\begin{abstract}{\small{
{\textbf{Abstract:}} The Deodhar decomposition of the Grassmannian is a refinement of the Schubert, Richardson, and positroid stratifications of the Grassmannian. Go-diagrams are certain fillings of Ferrers diagrams with black stones, white stones, and pluses which index Deodhar components in the Grassmannian.

We provide a series of corrective flips on diagrams which may be used to transform arbitrary fillings of Ferrers shapes with black stones, white stones, and pluses into a Go-diagram. This provides an extension of Lam and Williams' Le-moves for transforming reduced diagrams into Le-diagrams to the context of non-reduced diagrams.

Next, we address the question of describing when the closure of one Deodhar component is contained in the closure of another. We show that if one Go-diagram $D$ is obtained from another $D'$ by replacing certain stones with pluses, then applying corrective flips, that there is a containment of closures of the associated Deodhar components, $\overline{\calD'} \subset \overline{\calD}$.

Finally, we address the question of verifying whether an arbitrary filling of a Ferrers shape with black stones, white stones, and pluses is a Go-diagram. We show that no reasonable description of the class of Go-diagrams in terms of forbidden subdiagrams can exist by providing an injection from the set of valid Go-diagrams to the set of minimal forbidden subdiagrams for the class of Go-diagrams. In lieu of such a description, we offer an inductive characterization of the class of Go-diagrams.}}
\end{abstract}

\section{Introduction} \label{sec:intro}

The Deodhar decomposition of a flag manifold was introduced in \cite{deodhar:geometric_aspects} with the purpose of computing Kazhdan-Lusztig $R$-polynomials. Associated to each pair of permutations $u \leq v$ in Bruhat order is a Richardson cell in the flag manifold. Components in the Deodhar decomposition are indexed by certain subexpressions ${\mathbf{u}}$ for $u$ of an expression ${\mathbf{v}}$ for $v$ in the Coxeter generators. The Deodhar decomposition refines the Richardson decomposition, with the Richardson cell indexed by $u \leq v$ being the disjoint union of Deodhar components indexed by ${\mathbf{u}} \prec {\mathbf{v}}$. Deodhar components are homeomorphic to products of tori and affine spaces, and an explicit parameterization of Deodhar components is given in \cite{marsh:parametrizations}.

A Deodhar component in the Grassmannian is the projection of Deodhar component from the flag manifold to the Grassmannian. A positroid in the Grassmannian is the projection of a Richardson cell from the flag manifold to the Grassmannian. So, the Deodhar decomposition refines the positroid decomposition. In fact, \cite{marsh:parametrizations} shows that the Deodhar decomposition agrees with the positroid decomposition when restricted to the positive part of the Grassmannian and refines it away from the positive Grassmannian.

From any point in $A$ the Grassmannian, one may construct a soliton solution $u_A(t,x,y)$ to the KP differential equation. One may construct a contour plot of this soliton, and \cite{kodama:deodhar_decomposition} shows that when $t \ll 0$ this contour plot depends only on which Deodhar component $A$ lies in. In developing this theory, Kodama and Williams introduce Go-diagrams, certain fillings of a Ferrers shape with black stones, white stones, and pluses which index Deodhar components in the Grassmannian. When a Deodhar component intersects the positive Grassmannian, its Go-diagram is exactly the \Le-diagram indexing the positroid cell it agrees with. In \cite{lam:total_positivity}, a set of ``\Le-moves" is given which may be used to transform any reduced diagram into a \Le-diagram.

In this paper, we address the problems:

\begin{itemize}
\item[(1)] Provide a set of local moves which may be used to transform any, not necessarily reduced, diagram into a Go-diagram.
\item[(2)] Describe the boundary structure of Deodhar components in the Grassmannian.
\item[(3)] Given an arbitrary filling of a Ferrers shape with black stones, white stones, and pluses, provide a test for whether this diagram is a Go-diagram.
\end{itemize}

Section \ref{sec:corrective_flips} answers question (1), describing a set of {\textit{corrective flips}} which may be used to transform any diagram into a Go-diagram. Unlike the \Le-moves of \cite{lam:total_positivity}, it is possible to obtain more than one Go-diagram for a fixed starting diagram via corrective flips.

Section \ref{sec:boundaries} addresses question (2). In general, one does not expect questions of this form to have a reasonable answer. The Deodhar decomposition is known to not even be a stratification of the flag manifold, \cite{dudas:not_a_stratification}. However, Theorem \ref{thm:diagrammatic_boundaries} provides an intriguing class of boundaries, showing there is a containment of closures of Deodhar components $\overline{\calD'} \subset \overline{\calD}$ when the associated Go-diagrams $D'$ and $D$ are related by degenerating certain stones to pluses, then performing corrective flips. We conjecture that this theorem in fact provides a complete characterization of when there is a containment of closures of Deodhar components within a Schubert cell. Other aspects of the boundary structure of Deodhar components are explored in \cite{agarwala:nonorientable}.

Section \ref{sec:classification} addresses question (3). Ideally, one would like a description of Go-diagrams in terms of forbidden subdiagrams, analogous to the description of \Le-diagrams. We show that a reasonable description of this form cannot exist, by providing an injection from the set of valid Go-diagrams to the set of ``minimal forbidden subdiagrams" in Theorem \ref{thm:sad_injection}. So, the task of providing a list of forbidden subdiagrams for the class of Go-diagrams is at least as hard as providing a list of all valid Go-diagrams. In lieu of such a description, Theorem \ref{thm:characterization} provides an inductive characterization of Go-diagrams.

\section{Background and notation} \label{sec:background}

\subsection{The symmetric group}

Let $s_i$ denote the adjacent transposition $(i,i+1)$ in the symmetric group $\mathfrak{S}_n$. Italicized lowercase letters, $v$, will denote permutations and bold faced letters, ${\mathbf{v}}$, will denote specific expressions of permutations in the $s_i$'s. A {\textit{subexpression}} of $\mathbf{v}$ is a permutation obtained by replacing some of the factors in ${\mathbf{v}}$ by $\varepsilon$, the identity permutation in $\mathfrak{S}_n$. The terms ``expression" and ``word" will be used interchangeably.

Given an expression ${\mathbf{v}} = v_1 v_2 \cdots v_m$, let $v_{(i)} = v_1 v_2 \cdots v_{i}$ denote the product of the initial $i$ factors of ${\mathbf{v}}$. So, $v_{(0)} = \varepsilon$ and $v_{(m)} = v$. 

The {\textit{length}} of a permutation, $\ell(v)$, is the minimum number of letters in an expression of $v$. A word is {\textit{reduced}} if $\ell(v_{(i+1)}) = \ell(v_{(i)}) + 1$ for every $i$. All reduced words for a permutation contain the same number of factors. The {\textit{Bruhat order}} on permutations is the order given by setting $u \leq v$ if and only if some reduced word for $u$ is a subword of some reduced word for $v$.

A subexpression $\textbf{u}$ of $\textbf{v}$ is {\textit{distinguished}} if whenever $\ell(u_{(i)}v_{i+1}) < \ell(u_{(i)})$ , one also has $u_{i+1} = v_{i+1}$, (i.e. $u_{i+1} \neq \varepsilon$). Write $\textbf{u} \prec \textbf{v}$ if $\textbf{u}$ is a distinguished subexpression of $\textbf{v}$. The subexpression ${\textbf{u}}$ of ${\textbf{v}}$ is {\textit{positive}} if $\ell(u_{(i+1)}) \geq \ell(u_{(i)})$ for all $i$.

\begin{example}
Let ${\mathbf{v}} = s_1 s_2 s_1 s_3 s_2 s_1$. Then,
\begin{displaymath}
\varepsilon \varepsilon \varepsilon \varepsilon \varepsilon \varepsilon,
\mbox{ } s_1 \varepsilon s_1 \varepsilon \varepsilon \varepsilon,
\mbox{ and } s_1 \varepsilon \varepsilon \varepsilon \varepsilon s_1
\end{displaymath}
\noindent
are three subexpressions for the identity permutation in ${\mathbf{v}}$. The first is positive and distinguished, the second is distinguished but not positive, and the third is neither positive nor distinguished.
\end{example}

\begin{lemma}[Lemma 3.5 in \cite{marsh:parametrizations}] \label{lem:unique_positive_expression}
Let $u \leq v$ be permutations and $\textbf{v}$ be a reduced expression for $v$. Then, there is a unique positive distinguished subexpression for $u$ in $\textbf{v}$.
\end{lemma}

The Young subgroup $\mathfrak{S}_k \times \mathfrak{S}_{n-k} \subset \mathfrak{S}_n$ acts on a permutation $(v(1), v(2), \dots, v(n))$ by letting $\mathfrak{S}_k$ act on $(v(1), v(2), \dots, v(k))$ and $\mathfrak{S}_{n-k}$ act on $(v(k+1), v(k+2), \dots, v(n))$. Any coset in the quotient $\mathfrak{S}_n /  (\mathfrak{S}_k \times \mathfrak{S}_{n-k})$ has a unique representative of the form $(i_1, i_2, \dots, i_k, j_1, j_2, \dots, j_{n-k})$ where $i_1 < i_2 < \cdots < i_k$ and $j_1 < j_2 < \cdots < j_{n-k}$. These representatives are called {\textit{Grassmannian permutations}}. Grassmannian permutations are in bijection with subsets in $\binom{[n]}{k}$ sending $(i_1, i_2, \dots, i_k, j_1, j_2, \dots, j_{n-k})$ to $\{i_1, i_2, \dots, i_k\}$. Often, we will suppress curly braces and commas when writing sets to avoid unwieldy notation, writing $i_1 i_2 \dots i_k$ to mean $\{i_1, i_2, \dots, i_k\}$.

The Bruhat order on $\mathfrak{S}_n$ induces an order on the quotient $\mathfrak{S}_n /  (\mathfrak{S}_k \times \mathfrak{S}_{n-k})$ and thus a partial order on $\binom{[n]}{k}$. Concretely, if $I = \{i_1, i_2, \dots,  i_k\}$ and $J = \{j_1, j_2, \dots, j_k\}$ with $i_1 < i_2 < \cdots < i_k$ and $j_1 < j_2 < \cdots < j_k$, then $I \leq J$ if and only if $i_m \leq j_m$ for all $m$.

A {\textit{Ferrers shape}} is a collection of boxes obtained by taking a lattice path from the Northeast to Southwest corner of a $(n-k) \times k$ rectangle, then taking all boxes Northwest of this lattice path. The steps of the lattice path are labelled $1$ to $n$ starting at the Northeast corner. A box $b$ has coordinates $(i,j)$ if the vertical step of the boundary in the same row as $b$ is labelled $i$ and the horizontal step of the boundary in the same column as $b$ is labelled $j$. Ferrers shapes contained in an $(n-k) \times k$ box are in bijection with subsets $\binom{[n]}{k}$ sending the Ferrers shape $\lambda$ to the set $I_\lambda$ of labels of the vertical steps in its boundary path. Composing bijections, to the Ferrers shape $\lambda$ we also associate a Grassmannian permutation $v_{\lambda}$. Pictorially, the partial order on $\binom{[n]}{k}$ translates to containment of Ferrers shapes.

Given a box $b$ in a Ferrers diagram $D$, $b^{in}$ is the set of boxes weakly to the right and weakly below $b$. Additionally, set $b^{out} = D \setminus b^{in}$. We introduce a partial order on the boxes in a diagram, saying $c \preccurlyeq b$ if and only if $c \in b^{in}$.

A {\textit{pipe dream}} is a filling of a Ferrers shape with crossing tiles and elbow pieces,
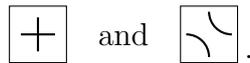
\begin{displaymath}
\begin{tikzpicture}
\begin{scope}[scale=.75]
\draw (0,0) -- (1,0) -- (1,1) -- (0,1) -- (0,0);
\draw[thick] (0.5,0.2) -- (0.5,0.8);
\draw[thick] (0.2,0.5) -- (0.8,0.5);
\draw (2,.5) node {and};
\draw (3,0) -- (4,0) -- (4,1) -- (3,1) -- (3,0);
\draw[thick] (3.5,.1) arc (0:90:.4cm);
\draw[thick] (3.5,.9) arc (180:270:.4cm);
\draw (4.2,.1) node {$.$};
\end{scope}
\end{tikzpicture}
\end{displaymath}
\noindent
Think of this filling as a collection of pipes flowing from the Southeast boundary to the Northwest boundary. From a pipe dream, we read off a permutation by labelling the edges along the North and West boundaries of the Ferrers shape such that for each pipe in the diagram, both ends of the pipe have the same label, then writing down the labels that appear along the Northwest boundary in order starting from the Northeast corner.

We say that two squares $b$ and $c$ in a pipe dream are a {\textit{crossing/uncrossing pair}} if two pipes cross in $b$, flow to the Northwest, then next uncross in $c$. A pipe dream is {\textit{reduced}} if it has no crossing/uncrossing pairs. Note that a crossing tile is a crossing if the label of the pipe entering from the bottom is larger than the label of the pipe entering from the right, and is an uncrossing otherwise. Pipe dreams were originally defined by Bergeron and Billey in \cite{bergeron:rc-graphs}, where they were called RC-graphs for ``reduced word, compatible sequence." They were later renamed pipe dreams by Knutson; we choose this terminology since the pipe dreams we consider will not in general be reduced.

\begin{example} \label{ex:pipe_dream}
The pipe dream
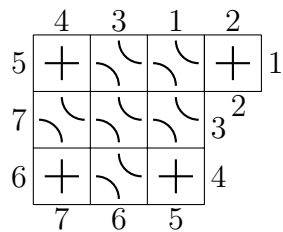
\begin{displaymath}
\begin{tikzpicture}
\begin{scope}[scale=.75]
%
\draw[step=1] (0,0) grid (3,3);
\draw[step=1] (3,2) grid (4,3);
%
\draw[thick] (0.5,0.2) -- (0.5,0.8);
\draw[thick] (0.2,0.5) -- (0.8,0.5);
\draw[thick] (1.5,.1) arc (0:90:.4cm);
\draw[thick] (1.5,.9) arc (180:270:.4cm);
\draw[thick] (2.5,0.2) -- (2.5,0.8);
\draw[thick] (2.2,0.5) -- (2.8,0.5);
%
\draw[thick] (.5,1.1) arc (0:90:.4cm);
\draw[thick] (.5,1.9) arc (180:270:.4cm);
\draw[thick] (1.5,1.1) arc (0:90:.4cm);
\draw[thick] (1.5,1.9) arc (180:270:.4cm);
\draw[thick] (2.5,1.1) arc (0:90:.4cm);
\draw[thick] (2.5,1.9) arc (180:270:.4cm);
%
\draw[thick] (0.5,2.2) -- (0.5,2.8);
\draw[thick] (0.2,2.5) -- (0.8,2.5);
\draw[thick] (1.5,2.1) arc (0:90:.4cm);
\draw[thick] (1.5,2.9) arc (180:270:.4cm);
\draw[thick] (2.5,2.1) arc (0:90:.4cm);
\draw[thick] (2.5,2.9) arc (180:270:.4cm);
\draw[thick] (3.5,2.2) -- (3.5,2.8);
\draw[thick] (3.2,2.5) -- (3.8,2.5);
%
\draw (4.25,2.5) node {$1$};
\draw (3.6,1.75) node {$2$};
\draw (3.25,1.35) node {$3$};
\draw (3.25,.5) node {$4$};
\draw (2.5,-.25) node {$5$};
\draw (1.5,-.25) node {$6$};
\draw (.5,-.25) node {$7$};
%
\draw (3.5,3.25) node {$2$};
\draw (2.5,3.25) node {$1$};
\draw (1.5,3.25) node {$3$};
\draw (.5,3.25) node {$4$};
\draw (-.25,2.5) node {$5$};
\draw (-.25,1.5) node {$7$};
\draw (-.25,.5) node {$6$};
\end{scope}
\end{tikzpicture}
\end{displaymath}
gives the permutation $(2,1,3,4,5,7,6)$. The squares $(4,5)$ and $(1,7)$ form a crossing/uncrossing pair.
\end{example}

A {\textit{$\circ/+$-diagram}} is a filling of a Ferrers shape with white stones and pluses,
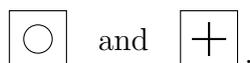
\begin{displaymath}
\begin{tikzpicture}
\begin{scope}[scale=.75]
\draw (0,0) -- (1,0) -- (1,1) -- (0,1) -- (0,0);
\draw[fill = white] (0.5,0.5) circle (.25);
\draw (2,.5) node {and};
\draw (3,0) -- (4,0) -- (4,1) -- (3,1) -- (3,0);
\draw[thick] (3.5,0.2) -- (3.5,0.8);
\draw[thick] (3.2,0.5) -- (3.8,0.5);
\draw (4.2,.1) node {$.$};
\end{scope}
\end{tikzpicture}
\end{displaymath}
\noindent
$\circ/+$-diagrams are in bijection with pipe dreams by replacing the circles with crossing tiles and the pluses with elbow pieces. This bijection is unfortunate, but is the standard convention in the literature.

A $\bullet/\circ/+$-diagram is a filling of filling of a Ferrers shape with black stones, white stones, and pluses,
\begin{displaymath}
\begin{tikzpicture}
\begin{scope}[scale=.75]
\begin{scope}[xshift=-2cm]
\draw (0,0) -- (1,0) -- (1,1) -- (0,1) -- (0,0);
\draw[fill = black] (0.5,0.5) circle (.25);
\draw (1.2,.1) node {$,$};
\end{scope}
\draw (0,0) -- (1,0) -- (1,1) -- (0,1) -- (0,0);
\draw[fill = white] (0.5,0.5) circle (.25);
\draw (1.2,.1) node {$,$};
\draw (2,.5) node {and};
\draw (3,0) -- (4,0) -- (4,1) -- (3,1) -- (3,0);
\draw[thick] (3.5,0.2) -- (3.5,0.8);
\draw[thick] (3.2,0.5) -- (3.8,0.5);
\draw (4.2,.1) node {$.$};\end{scope}
\end{tikzpicture}
\end{displaymath}
$\bullet/\circ/+$-diagrams are mapped to pipe dreams by sending both the black and white stones to crossing tiles, and sending the pluses to elbow tiles. So, $\bullet/\circ/+$-diagrams may be viewed as $\circ / +$-diagrams where the stones have been decorated to have two colors. Often, but not always, we will require that stones be colored black if and only if they are mapped to uncrossing tiles in the pipe dream. We state whether or not we make this assumption at the start of each section.

Label the top left box of a Ferrers shape contained in a $k \times (n-k)$ box with the simple transposition $s_{n-k}$. If the box to the left of $b$ is labelled with the transpositions $s_i$, label $b$ with $s_{i-1}$ and if the box above $b$ is labelled with $s_i$ label $b$ with $s_{i+1}$. Observe that for any box $b$, the permutation corresponding to the $\circ/+$ diagram where $b$ is filled with a white stone and all other boxes are filled with pluses is exactly the transposition labelling $b$. We use $s_b$ to denote the simple transposition labelling the box $b$.

A {\textit{reading order}} on a Ferrers shape of shape $\lambda$ containing $m$ boxes is a filling of the boxes with the integers from $1$ to $m$ which is increasing upward and to the left. Reading the transpositions decorating the boxes of the Ferrers diagram in any reading order yields a reduced expression $\textbf{v}$ for $v_\lambda$. Reading only the transpositions decorating boxes containing stones in either a $\circ/+$ or a $\bullet/\circ/+$-diagram diagram in the same reading order gives a subexpression ${\mathbf{u}}$ of ${\mathbf{v}}$ for the permutation given by the associated pipe dream.

\begin{theorem}[Proposition 4.5 in \cite{lam:total_positivity}] \label{thm:reading_order_doesnt_matter}
Let $D$ be a $\circ/+$ or $\bullet/\circ/+$-diagram giving the subword, word pair $\textbf{u}, \textbf{v}$ in some reading order.
\begin{itemize}
\item[(i)] The permutations $v$, coming from the Ferrers shape, and $u$, coming from the pipe dream depend only on $D$, not the choice of reading order.
\item[(ii)] Whether $\textbf{u}$ is a distinguished subexpression of $\textbf{v}$ depends only on $D$, not the choice of reading order.
\item[(iii)] Whether $\textbf{u}$ is a positive subexpression of $\textbf{v}$ depends only $D$, not on the choice of reading order.
\end{itemize}
\end{theorem}

This theorem is proved by noting that if the expressions $\textbf{u}$ and $\textbf{u}'$ are obtained by altering the reading order on the same diagram, then they are related by commutations of the Coxeter generators, and thus $u = u'$.

Let $b$ be a box in the $\circ/+$ or $\bullet/\circ/+$-diagram $D$. Define $u_{b^{in}}^{D}$ to be the permutation obtained by multiplying the transpositions labelling all boxes containing stones in $b^{in}$ in $D$ some valid reading order. As a corollary of Theorem \ref{thm:reading_order_doesnt_matter}, $u_{b^{in}}^{D}$ does not depend on the choice of reading order. If the diagram is clear from context, we will simply write $u_{b^{in}}$ instead of $u_{b^{in}}^D$. Let
\begin{displaymath}
u_b^D =
\begin{cases}
s_b & \mbox{if $b$ contains a stone in $D$,} \\
\varepsilon & \mbox{if $b$ contains a plus in $D$.}
\end{cases}
\end{displaymath}
\noindent
We will simply write $u_b$ instead of $u_b^D$ if the diagram is clear from context.

\begin{definition} \label{def:go-diagram}
Let $D$ be a $\bullet/\circ/+$-diagram.  Then, $D$ is a {\textit{Go-diagram}} if and only if a box $b$ contains a black stone if and only if $\ell(u_{b^{in}} u_b s_b) < \ell(u_{b^{in}} u_b)$.
\end{definition}

Let ${\mathbf{u}}, {\mathbf{v}}$ be the subword, word pair associated to the diagram $D$ in some reading order. Theorem \ref{thm:reading_order_doesnt_matter} implies that if $D$ is a Go-diagram, then ${\mathbf{u}}$ is a distinguished subexpression of ${\mathbf{v}}$, and that this is independent of the choice of reading order. If a box $b$ has the property that $\ell(u_{b^{in}} u_b s_b ) < \ell(u_{b^{in}} u_b)$, but $b$ is not filled with a black stone, we say that $b$ {\textit{violates the distinguished property}}.

\begin{definition} \label{def:le-diagram}
Let $D$ be a $\bullet/\circ/+$-diagram. Then, $D$ is a {\textit{\Le-diagram}} is any of the following equivalent criterion hold:
\begin{itemize}
\item[(i)] $D$ is an Go-diagram and contains no black stones.
\item[(ii)] Let ${\mathbf{u}}, {\mathbf{v}}$ be the subword, word pair associated to $D$. Then, ${\mathbf{u}}$ is a positive distinguished subexpression of ${\mathbf{v}}$
\item[(iii)] $D$ contains no black stones, and there is no box $b$ in $D$ containing a white stone such that there is a plus to the left of $b$ in its column and above $b$ in its row.
\end{itemize}
\end{definition}

The equivalence between points (i) and (ii) is immediate. The equivalence between points (ii) and (iii) is Theorem 5.1 in \cite{lam:total_positivity}. The condition in point (iii) is called the {\textit{\Le-property}}. The symbol ``\Le" is pronounced ``le," the backward spelling of the letter ``el."

\begin{example}
Consider the diagrams
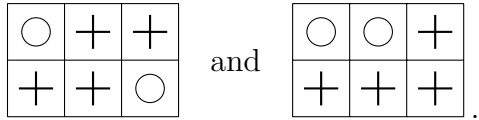
\begin{displaymath}
\begin{tikzpicture}
\begin{scope}[scale=.75]
\draw[step=1] (0,0) grid (3,2);
%
\draw[thick] (.5,.2) -- (.5,.8);
\draw[thick] (.2,.5) -- (.8,.5);
\draw[thick] (1.5,.2) -- (1.5,.8);
\draw[thick] (1.2,.5) -- (1.8,.5);
\draw[fill = white] (2.5,0.5) circle (.25);
%
\draw[fill = white] (.5,1.5) circle (.25);
\draw[thick] (1.5,1.2) -- (1.5,1.8);
\draw[thick] (1.2,1.5) -- (1.8,1.5);
\draw[thick] (2.5,1.2) -- (2.5,1.8);
\draw[thick] (2.2,1.5) -- (2.8,1.5);
\draw (4,1) node {and};
\begin{scope}[shift={(5,0)}]
\draw[step=1] (0,0) grid (3,2);
%
\draw[thick] (.5,.2) -- (.5,.8);
\draw[thick] (.2,.5) -- (.8,.5);
\draw[thick] (1.5,.2) -- (1.5,.8);
\draw[thick] (1.2,.5) -- (1.8,.5);
\draw[thick] (2.5,.2) -- (2.5,.8);
\draw[thick] (2.2,.5) -- (2.8,.5);
%
\draw[fill = white] (.5,1.5) circle (.25);
\draw[fill = white] (1.5,1.5) circle (.25);
\draw[thick] (2.5,1.2) -- (2.5,1.8);
\draw[thick] (2.2,1.5) -- (2.8,1.5);
\draw (3.2,0) node {$.$};
\end{scope}
\end{scope}
\end{tikzpicture}
\end{displaymath} 
\noindent
The diagram on the left corresponds to the subexpression $s_2 \varepsilon \varepsilon \varepsilon \varepsilon s_3$ of $s_2 s_3 s_4 s_1 s_2 s_3$ and the diagram on the right corresponds to the subexpression $\varepsilon \varepsilon \varepsilon \varepsilon s_2 s_3$. The diagram on the left is not a \Le-diagram, which can be seen by noting that its subexpression is not distinguished or by noting that the box $(2,3)$ contains a white stone, but has a plus both above it and to its left.
\end{example}

\begin{proposition}
The locations of only the white stones or of only the pluses are enough to uniquely determine a Go-diagram.
\end{proposition}

\begin{proof}
Given a Ferrers shape filled with white stones, we complete the filling in increasing order in the partial order on boxes. If $c$ is some box in the diagram and all boxes in $c^{in} \setminus c$ have been filled, we may compute the permutation $u_{c^{in}}u_c$, which does not depend on the filling of $c$. Then, fill $c$ with a plus if $\ell(u_{c^{in}} u_cs_c) > \ell(u_{c^{in}} u_c)$ and fill $c$ with a black stone otherwise.

Given a Ferrers shape filled with pluses, construct a pipe dream by placing elbow tiles in the squares containing pluses and crossing tiles in all other squares. Then, for each square not containing a plus in the Ferrers shape, fill it with a white stone if the square is a crossing in the pipe dream and a black stone if it is an uncrossing.
\end{proof}

The following propositions follows directly from the definitions.

\begin{proposition} \label{prop:length_from_filling}
Let $D$ be the Deodhar component associated to the distinguished subexpression ${\mathbf{u}} \prec {\mathbf{v}}$. Then,
\begin{displaymath}
\begin{split}
\ell(u) & = \#(\mbox{of $\circ$'s in $D$}) - \#(\mbox{of $\bullet$'s in $D$}) \\
& = \ell(v) - \#(\mbox{of $+$'s in $D$}) - 2 \cdot \#(\mbox{of $\bullet$'s in $D$}).
\end{split}
\end{displaymath}
\end{proposition}

\subsection{The Deodhar decomposition of the Grassmannian}

Let $Gr(k,n)$ be the Grassmannian of $k$ dimensional subspaces of $\rr^n$. For $I \in \binom{[n]}{k}$, let $\Delta_I$ be the $I^{th}$ Pl\"ucker coordinate on $Gr(k,n)$. That is, choosing a basis on $\rr^n$, if $V \in Gr(k,n)$ is presented as the row span of a $k \times n$ matrix, then $\Delta_I(V)$ is the maximal minor of this matrix using columns labelled by $I$. The Pl\"ucker coordinates give an embedding of $Gr(k,n)$ into the projective space $\mathbb{P}^{\binom{n}{k} - 1}$. The Pl\"ucker embedding of the Grassmannian is the subvariety of $\mathbb{P}^{\binom{n}{k} - 1}$ defined by the Pl\"ucker relations
\begin{equation} \label{eqn:plucker_relation}
\Delta_I \Delta_J = \sum_{j \in J} \pm \Delta_{I \setminus i \cup j} \Delta_{J \setminus j \cup i},
\end{equation}
\noindent
for each pair of sets $I, J \in \binom{[n]}{k}$ and each element $i \in I$. The signs in the sum on the right hand side are obtained by ordering the indices of $I$ and $J$ in increasing order, placing $j$ in $i$'s old position in $I$ and $i$ in $j$'s old position in $J$, then taking $-1$ times the product of the signs of the two permutations need to rearrange to new sets so that their elements are in increasing order. For example,
\begin{displaymath}
\Delta_{123} \Delta_{245} = \Delta_{234} \Delta_{125} - \Delta_{235} \Delta_{124}
\end{displaymath}
\noindent
is a Pl\"ucker relation.


The positive part of the Grassmannian, $Gr_{\geq 0}(k,n)$, is the subset of $Gr(k,n)$ where all Pl\"ucker coordinates have the same sign.

Deodhar components in the Grassmannian are semialgebraic subsets obtained by setting some Pl\"ucker coordinates equal to zero and demanding other Pl\"ucker coordinates do not vanish. This description of Deodhar components appears as Theorem 7.8 in \cite{talaska:network_parameterizations}. Deodhar components were originally defined in \cite{deodhar:geometric_aspects}.

\begin{definition} \label{def:set_I_b}
Let $b$ be a box in the Go-diagram $D$ associated to the distinguished subword pair ${\mathbf{u}} \prec {\mathbf{v}}$. Define
\begin{equation} \label{eqn:I_b}
I_b = u_{b^{in}} u_b (v_{b^{in}})^{-1} v \{n, n-1, \dots, n-k+1\}.
\end{equation}
\end{definition}

Diagrammatically, $I_b$ may be computed by:
\begin{itemize}
\item Changing the filling of all boxes in $b^{out}$ to white stones,
\item Changing the filling of $b$ to a plus,
\item Computing the pipe dream associated to this diagram,
\item Setting $I_b$ to be the labels of the pipes appearing along the left boundary of this pipe dream.
\end{itemize}

A description of the set $I_b$ is given in \cite{agarwala:nonorientable} using a network associated to the Go-diagram defined in \cite{talaska:network_parameterizations}. Proposition \ref{prop:I_b_again} tells how to recover the sets $I_b$ given the Deodhar component.

\begin{example} \label{ex:sets_from_boxes}
A Go-diagram $D$ is pictured on the left. In the Ferrers shape on the right, each box $b$ is labelled with the set $I_b$ from $D$.
\begin{displaymath}
\begin{tikzpicture}
\begin{scope}[scale=.75]
\draw[step = 1] (0,0) grid (3,3);
%
%
\draw[thick] (0.5,0.2) -- (0.5,0.8);
\draw[thick] (0.2,0.5) -- (0.8,.5);
\draw[fill = white] (1.5,0.5) circle (.25);
\draw[thick] (2.5,0.2) -- (2.5,0.8);
\draw[thick] (2.2,0.5) -- (2.8,0.5);
%
\draw[fill = black] (0.5,1.5) circle (.25);
\draw[thick] (1.5,1.2) -- (1.5,1.8);
\draw[thick] (1.2,1.5) -- (1.8,1.5);
\draw[fill = white] (2.5,1.5) circle (.25);
%
\draw[thick] (0.5,2.2) -- (0.5,2.8);
\draw[thick] (0.2,2.5) -- (0.8,2.5);
\draw[fill = black] (1.5,2.5) circle (.25);
\draw[thick] (2.5,2.2) -- (2.5,2.8);
\draw[thick] (2.2,2.5) -- (2.8,2.5);
\end{scope}
\begin{scope}[xshift=3.5cm,yshift=-.375cm]
\draw[step=1] (0,0) grid (3,3);
\draw (2.5,.5) node {124};
\draw (1.5,.5) node {125};
\draw (.5,.5) node {126};
\draw (2.5,1.5) node {134};
\draw (1.5,1.5) node {145};
\draw (.5,1.5) node {146};
\draw (2.5,2.5) node {234};
\draw (1.5,2.5) node {245};
\draw (.5,2.5) node {456};
\end{scope}
\end{tikzpicture}
\end{displaymath}
\end{example}

\begin{definition}[Theorem 7.8 in \cite{talaska:network_parameterizations}] \label{thm:plucker_coordinates}
Let $D$ be a Go-diagram of shape $\lambda$. Then, the {\textit{Deodhar component}} $\calD$ associated to $D$ is the subset in $Gr(k,n)$ defined by:
\begin{itemize}
\item $\Delta_{I_b} = 0$ for all boxes $b \in D$ containing white stones.
\item $\Delta_{I_b} \neq 0$ for all boxes $b \in D$ containing pluses.
\item $\Delta_{I_\lambda} \neq 0$.
\item $\Delta_{S} = 0$ for all $S \ngeq I_\lambda$.
\end{itemize}
\end{definition}

In general, we will use upper case letters to refer to Go-diagrams and calligraphic letters to refer to Deodhar components. When we want to make explicit reference to the distinguished subword pair associated to a Go-diagram, we will use the notation $\calD_{{\mathbf{u}},{\mathbf{v}}}$. The Deodhar component associated to the Go-diagram in Example \ref{ex:sets_from_boxes} is the subset of $Gr(3,6)$ where
\begin{displaymath}
\begin{split}
& \Delta_{134}, \Delta_{125} = 0, \mbox{and} \\
& \Delta_{123}, \Delta_{124}, \Delta_{126}, \Delta_{145}, \Delta_{234}, \Delta_{456} \neq 0.
\end{split}
\end{displaymath}

\begin{theorem}[Theorem 1.1 in \cite{deodhar:geometric_aspects}] \label{thm:parameterization_of_deodhar}
Let $\calD$ be the Deodhar component labelled by the Go-diagram $D$. Then, $\calD$ is homeomorphic to
\begin{displaymath}
\rr^{\#\left(\bullet\mathit{'s} \ \mathit{in} \ D \right)} \times \left( \rr \setminus \{0\}\right)^{\#\left(+\mathit{'s} \ \mathit{in} \ D \right)}.
\end{displaymath}
\end{theorem}

The following interpretation of the sets $I_b$ is given in Theorem 1.17 in \cite{agarwala:nonorientable}.

\begin{proposition} \label{prop:I_b_again}
Suppose the vertical steps of the boundary of the Ferrers shape $\lambda$ are $i_1 < i_2 < \cdots < i_k$ and the horizontal steps are $j_1 < j_2 < \cdots < j_{n-k}$. Let $b = (i_\ell,j_m)$ be a box in the Go-diagram $D$ of shape $\lambda$ indexing the Deodhar component $\calD$. Let $I'_b = I_b \setminus j_m \cup i_\ell$. Then, $I'_b$ is the maximal set such that $i_1, i_2, \dots i_\ell \in I'_b$, $j_n, j_{n-1}, \dots, j_{m} \notin I'_b$, and $\Delta_{I'_b}$ is not uniformly vanishing on $\calD$.
\end{proposition}

The following proposition is an immediate consequence of Corollary 3.16 in \cite{agarwala:nonorientable}, together with the fact that the Deodhar decomposition refines the Richardson decomposition, Theorem \ref{thm:deodhar_decomposes_richardson} below.

\begin{proposition} \label{prop:large_sets_vanish}
Suppose the vertical steps of the boundary of the Ferrers shape $\lambda$ are $i_1 < i_2 < \cdots < i_k$ and the horizontal steps are $j_1 < j_2 < \cdots < j_{n-k}$. Let $b = (i_\ell,j_m)$ be a box in the Go-diagram $D$ of shape $\lambda$ indexing the Deodhar component $\calD$. Let
\begin{equation} \label{eqn:J_b}
J_b =
u_{b^{in}} (v_{b^{in}})^{-1} v \{n, n-1, \dots, n-k+1\}.
\end{equation}
\noindent
If $S \in \binom{[n]}{k}$ is such that $i_1, i_2, \dots, i_{\ell -1} \in S$, $j_n, j_{n-1}, \dots, j_{m+1} \notin S$, and $S > J_b$, then $\Delta_S$ vanishes uniformly on $\calD$.
\end{proposition}

\begin{proposition}
Suppose that boxes $b$ and $c$ in a Go-diagram $D$ share an edge. Then, $I_b$ and $I_c$ differ by a single element.
\end{proposition}

\begin{proof}
Let $i_1 < i_2 < \cdots < i_k$ be the labels of the vertical steps of the boundary of $D$ and $j_1< j_2 < \dots < j_{n-k}$ be the labels of the horizontal steps of the boundary of $D$ and suppose $b = (i_\ell,j_m)$. Let $p$ be the label of the pipe entering $b$ from the bottom and $q$ be the label of the pipe entering $b$ from the left in the pipe dream associated to $D$. If $c =  (i_{\ell}, j_{m+1})$, then $I_c = I_b \setminus p \cup j_{m+1}$. If $c = (i_{\ell - 1}, j_m)$, then $I_c = I_b \setminus i_{\ell - 1} \cup q$. 
\end{proof}

\subsection{Other decompositions of the Grassmannian}

This section briefly remarks about how the Deodhar decomposition is related to other common decompositions of the Grassmannian.

Let $v \in \mathfrak{S}_n$ be the Grassmannian permutation associated to Ferrers shape $\lambda$. Associated to pairs of permutations $u, v$ with various constraints imposed on $u$, there are several decompositions of the Grassmannian $Gr(k,n)$. The stricter the constraint imposed on $u$, the coarser the decomposition of the Grassmannian. Below are the common decompositions of $G(k,n)$ and the associated constraints on $u$ arranged from coarsest to finest.

\begin{displaymath}
\begin{array}{|c|c|c|}
\hline
\mbox{Decomposition} & \mbox{Notation} & \mbox{Constraints on $u$} \\\hline\hline
\mbox{Schubert} & \calS_{v} & u = \varepsilon \\\hline
\mbox{Richardson} & \calR_{u,v} & u \mbox{ is Grassmannian}, u \leq v \\\hline
\mbox{Positroid} & \calP_{u,v} &  u \leq v \\\hline
\mbox{Deodhar} &  \calD_{{\mathbf{u}},{\mathbf{v}}} & {\mathbf{u}} \prec {\mathbf{v}} \\\hline
\end{array}
\end{displaymath}

The Deodhar decomposition differs from the other decompositions in this list in that it doesn't just care that $u \leq v$, but how $\mathbf{u}$ is presented as a subword of $\mathbf{v}$.

All of the components in these decompositions have the feature that they can be described as subsets of the Grassmannian by setting certain Pl\"ucker coordinates to zero, demanding certain other Pl\"ucker coordinates be non-zero, and leaving the remaining Pl\"ucker coordinates unspecified. The coarser the decomposition, the more Pl\"ucker coordinates are left unspecified. From a combinatorial standpoint, all of the decompositions can be described as introducing decorations to the Ferrers shape $\lambda$.

\begin{itemize}
\item The Schubert stratification remembers only $\lambda$. This should be viewed coming from the fact that the $\circ/+$-diagram corresponding to the positive distinguished subexpression of identity permutation in $\mathbf{v}$ is $\lambda$ with every square with a plus. Let $I_{\lambda} \in \binom{[n]}{k}$ be the set associated to $v$. The Schubert cell $\calS_v$ is defined by $\Delta_{I_\lambda} \neq 0$ and $\Delta_{S} = 0$ for all $S \ngeq I_\lambda$.

\item The Richardson stratification introduces another Ferrers shape $\mu$ contained inside $\lambda$; the pair is often called a skew shape. This should be viewed coming from the fact that the $\circ/+$-diagram corresponding to the positive distinguished subexpression $u$ in $\mathbf{v}$ can be drawn by first drawing the shape $\mu$ associated $u$ inside $\lambda$, then filling all squares in the skew diagram $\lambda/\mu$ with pluses and filling all other squares with white stones. The Richardson cell $\calR_{u,v}$ is defined by  $\Delta_{I_\mu}, \Delta_{I_\lambda} \neq 0$ and $\Delta_{S} = 0$ for all $S \nleq I_\mu$ and all $S \ngeq I_\lambda$.

\item The positroid stratification\footnote{Positroids were originally defined to stratify positive Grassmannian. There have been several extensions of this stratification to the entire Grassmannian. When we say ``positroid strata," we mean the stratification of the Grassmannian by projections of Richardson varieties in the full flag manifold, studied in \cite{knutson:juggling}.} introduces the filling of $\lambda$ with pluses and white stones corresponding to the unique positive distinguished subexpression for $u$ in $\mathbf{v}$. That is, positroid cells are indexed by \Le-diagrams. We will not need an explicit description of the positroid cell $\calP_{u,v}$. These cells have been studied extensively and a description may be found in \cite{postnikov:total_positivity}.

\item The Deodhar decomposition is associated to the Go-diagram built from the pair $\mathbf{u} \prec \mathbf{v}$, as we saw in the previous section.
\end{itemize}

As a word of caution, we remark that while \Le-diagrams are in general Go-diagrams, the positroid cell and Deodhar component associated to the same diagram are in general different. For example, consider the following \Le-diagram.

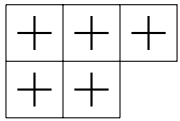
\begin{displaymath}
\begin{tikzpicture}
\begin{scope}[scale = .75]
\draw[step=1] (0,0) grid (2,2);
\draw (2,1) -- (3,1) -- (3,2) -- (2,2);
\draw[thick] (.5,.2) -- (.5,.8);  \draw[thick] (.2,.5) -- (.8,.5);
\draw[thick] (1.5,.2) -- (1.5,.8);  \draw[thick] (1.2,.5) -- (1.8,.5);
\draw[thick] (.5,1.2) -- (.5,1.8);  \draw[thick] (.2,1.5) -- (.8,1.5);
\draw[thick] (1.5,1.2) -- (1.5,1.8);  \draw[thick] (1.2,1.5) -- (1.8,1.5);
\draw[thick] (2.5,1.2) -- (2.5,1.8);  \draw[thick] (2.2,1.5) -- (2.8,1.5);
\end{scope}
\end{tikzpicture}
\end{displaymath}
\noindent
The positroid starta associated to this diagram is determined by
\begin{displaymath}
\Delta_{13}, \Delta_{23}, \Delta_{34}, \Delta_{45}, \Delta_{51} \neq 0; \quad \Delta_{12} = 0,
\end{displaymath}
\noindent
while the Deodhar component additionally requires that $\Delta_{14} \neq 0$. This \Le-diagram indexes the pair of permutations $\varepsilon, (2,3,4,1,3)$. So, it may additionally be seen as determining either a Schubert cell or a Richardson cell. The Schubert cell associated to this pair of permutations is determined by
\begin{displaymath}
\Delta_{13} \neq 0; \quad \Delta_{12} = 0,
\end{displaymath}
\noindent
and the Richardson cell is determined by
\begin{displaymath}
\Delta_{13}, \Delta_{45} \neq 0; \quad \Delta_{12} = 0.
\end{displaymath}
\noindent
In general, if a diagram could serve as an index for two decompositions, the set of vanishing Pl\"ucker coordinates will be the same in both decompositions. The set of nonvanishing Pl\"ucker coordinates in the coarser decomposition will be a subset of the set of nonvanishing Pl\"ucker coordinates in the finer decomposition.

Deodhar components were originally introduced to refine the Richardson stratification.

\begin{theorem}[Corollary 1.2 in \cite{deodhar:geometric_aspects}] \label{thm:deodhar_decomposes_richardson}
Let $\calR_{u,v}$ be a Richardson cell. Then,
\begin{displaymath}
\calR_{u,v} = \bigsqcup_{u' \sim u} \bigsqcup_{{\mathbf{u'}} \prec {\mathbf{v}}} \calD_{{\mathbf{u'}},{\mathbf{v}}},
\end{displaymath}
\noindent
where the first union is across all $u'$ in the same equivalence class as $u$ in $\mathfrak{S}_n / (\mathfrak{S}_{k} \times \mathfrak{S}_{n - k})$ and the second union is across all distinguished subexpressions for $u'$ in ${\mathbf{v}}$.
\end{theorem}

Since the Grassmannian is the disjoint union of its Richardson cells, the Grassmannian is then the disjoint union of its Deodhar components. Though Deodhar components predate Lusztig's notion of positivity in flag varieties, \cite{lusztig:total_positivity}, and Postnikov's concrete description of positroid cells in $Gr_{\geq 0}(k,n)$, \cite{postnikov:total_positivity}, the Deodhard decomposition does refine the positroid stratification and in fact agrees with it when restricted to the positive part of the Grassmannian.

\begin{theorem}[Lemma 11.6 in \cite{marsh:parametrizations}] \label{thm:positive_deodhar_components}
Let $\calD$ be the Deodhar component in $Gr(k,n)$ labelled by the Go-diagram $D$. Then, $\calD \cap Gr_{\geq 0}(k,n)$ is nonempty if and only if $D$ is a \Le-diagram. In this case, let $\calP$ be the positroid cell labelled by $D$. Then,
\begin{displaymath}
\calD \cap Gr_{\geq 0}(k,n) = \calP \cap Gr_{\geq 0}(k,n).
\end{displaymath}
\noindent
Further,
\begin{displaymath}
\calP_{u,v} = \bigsqcup_{{\mathbf{u}} \prec {\mathbf{v}}} \calD_{{\mathbf{u}}, {\mathbf{v}}},
\end{displaymath}
\noindent
where the union is across all distinguished subexpressions for $u$ in ${\mathbf{v}}$.
\end{theorem}

If one wants, they may take Theorem \ref{thm:positive_deodhar_components} as the definition of the positroid strata $\calP_{u,v}$, since we did not explicitly provide this definition.

\section{Corrective flips} \label{sec:corrective_flips}

Throughout this section, stones in diagrams will be colored black if and only if they are uncrossings.

In \cite{lam:total_positivity}, Lam and Williams address the problem of giving a series of local moves on $\circ/+$-diagrams which may be used to transform any reduced diagram into the \Le-diagram corresponding to the same pair of permutations. Such moves are called \Le-moves. They solve this problem in all cominuscule types. In type A, the only case we consider in this paper, \Le-moves are moves of the following form.
\begin{equation} \label{eqn:le-move}
\begin{tikzpicture}
\begin{scope}[scale=.75]
\draw (0,0) -- (4,0) -- (4,3) -- (0,3) -- (0,0);
\draw (1,0) -- (1,1) -- (0,1);
\draw (1,3) -- (1,2) -- (0,2);
\draw (3,3) -- (3,2) -- (4,2);
\draw (3,0) -- (3,1) -- (4,1);
\draw[thick] (0.5,0.2) -- (0.5,0.8);
\draw[thick] (0.2,0.5) -- (0.8,0.5);
\draw[thick] (0.5,2.2) -- (0.5,2.8);
\draw[thick] (0.2,2.5) -- (0.8,2.5);
\draw[thick] (3.5,2.2) -- (3.5,2.8);
\draw[thick] (3.2,2.5) -- (3.8,2.5);
\draw[fill = white] (3.5,.5) circle (.25);

\draw (2,1.5) node {white stones};
\draw[thick, ->] plot [smooth, tension=.5] coordinates {(4.5,1.5) (4.75,1.65) (5,1.7) (5.25,1.65) (5.5,1.5)};
\begin{scope}[xshift=6cm]
\draw (0,0) -- (4,0) -- (4,3) -- (0,3) -- (0,0);
\draw (1,0) -- (1,1) -- (0,1);
\draw (1,3) -- (1,2) -- (0,2);
\draw (3,3) -- (3,2) -- (4,2);
\draw (3,0) -- (3,1) -- (4,1);
\draw[thick] (0.5,0.2) -- (0.5,0.8);
\draw[thick] (0.2,0.5) -- (0.8,0.5);
\draw[thick] (3.5,0.2) -- (3.5,0.8);
\draw[thick] (3.2,0.5) -- (3.8,0.5);
\draw[thick] (3.5,2.2) -- (3.5,2.8);
\draw[thick] (3.2,2.5) -- (3.8,2.5);
\draw[fill = white] (.5,2.5) circle (.25);
\draw (2,1.5) node {white stones};
\end{scope}
\end{scope}
\end{tikzpicture}
\end{equation}
\noindent

The following theorem collects Lemma 4.13, Proposition 4.14, and Theorem 5.3 from \cite{lam:total_positivity}.

\begin{theorem} \label{thm:positroids_via_le_moves}
Let $D$ be a reduced $\circ/+$-diagram.
\begin{itemize}
\item[(i)] If $D'$ is obtained from $D$ via \Le-moves, the associated permutations associated to $D$ and $D'$ are identical.
\item[(ii)] $D$ is a \Le-diagram if and only if no \Le-moves may be applied to it.
\item[(iii)] Any sequence of \Le-moves applied to $D$ terminates in the unique \Le-diagram associated to the same pair of permutations.
\end{itemize}
\end{theorem}

The goal of this section is to provide an analogous set of moves to transform any $\bullet/\circ/+$-diagram into a Go-diagram. The following example shows that \Le-moves are not sufficient to transform any reduced diagram into a Go-diagram, so additional moves really are needed.

\begin{example}
Consider the following diagram.
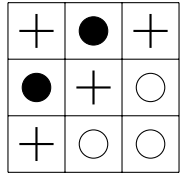
\begin{displaymath}
\begin{tikzpicture}
\begin{scope}[scale=.75]
\draw[step=1] (0,0) grid (3,3);
%
\draw[thick] (.5,.2) -- (.5,.8);
\draw[thick] (.2,.5) -- (.8,.5);
\draw[fill = white] (1.5,.5) circle (.25);
\draw[fill = white] (2.5,.5) circle (.25);
%
\draw[fill = black] (.5,1.5) circle (.25);
\draw[thick] (1.5,1.2) -- (1.5,1.8);
\draw[thick] (1.2,1.5) -- (1.8,1.5);
\draw[fill = white] (2.5,1.5) circle (.25);
%
\draw[thick] (.5,2.2) -- (.5,2.8);
\draw[thick] (.2,2.5) -- (.8,2.5);
\draw[fill = black] (1.5,2.5) circle (.25);
\draw[thick] (2.5,2.2) -- (2.5,2.8);
\draw[thick] (2.2,2.5) -- (2.8,2.5);
\end{scope}
\end{tikzpicture}
\end{displaymath}
\noindent
This diagram is not a Go-diagram; the square in the top left corner violates the distinguished property since it should be an uncrossing with the square in the bottom right corner. Further, there are no \Le-moves which may be applied to this diagram.
\end{example}

\begin{definition} \label{def:corrective_flip}
Let $D$ be a $\bullet/\circ/+$-diagram. Given a plus which violates the distinguished property, a {\textit{corrective flip}}:
\begin{itemize}
\item[(i)] switches the plus with either the white stone with which it violates the distinguished property or that stone's uncrossing partner if it exists, then
\item[(ii)] relabels the stones in the diagram so that a stone is black if and only if it is an uncrossing.
\end{itemize}
\end{definition}

\begin{remark}
In \cite{lam:total_positivity}, one of the defining features of \Le-moves is that only two squares change filling during the \Le-move, one from a white stone to a plus and the other from a plus to a white stone. In the case of a $\bullet/\circ/+$-diagram, the coloring of stones white or black should be thought of purely as a pneumonic for which stones correspond to crossings and uncrossings in the pipe dream. When performing a corrective flip, in the pipe dream only two tiles change, one from an elbow piece to a crossing and the other from a crossing to an elbow. The possible change in coloring of other stones in the diagram is a necessary side effect of this two square swap.
\end{remark}

\begin{proposition} \label{prop:corrective_flip_observations}
Let $D$ be a $\bullet/\circ/+$-diagram.
\begin{itemize}
\item[(i)] $D$ is a Go-diagram if and only if there are no available corrective flips.

\item[(ii)] Corrective flips preserve the pair of permutations associated to a diagram.

\item[(iii)] Corrective flips preserve number of black stones, white stones, and pluses in a diagram.

\item[(iv)] Suppose pipes $i$ and $j$ cross at the crossing tile involved in a corrective flip. Then, the only stones which change color when preforming a corrective flip are along pipes $i$ and $j$ on the segments between the plus and crossing tile involved in the flip.

\item[(v)] $D$ can be transformed into a Go-diagram via corrective flips.
\end{itemize}
\end{proposition}

\begin{proof}
Points (i), (ii), and (iv) are obvious looking at the pipe dream associated to a diagram. Point (iii) is a consequence of Proposition \ref{prop:length_from_filling} and the fact that a corrective flip preserves the number of pluses in a diagram. For point (v), observe that if we only preform corrective flips switching pluses and white stones, the pluses only move downward. We may preform such flips until no more corrective flips are available, at which point point (i) implies the end result is a Go-diagram.
\end{proof}

\begin{lemma} \label{lem:corrective_flips_decrease_length}
Let $D$ be a $\bullet/\circ/+$-diagram and let $D'$ be obtained from $D$ by performing a corrective flip. Then, $u_{b^{in}}^{D} \geq u_{b^{in}}^{D'}$ for all $b \in D$. 
\end{lemma}

\begin{proof}
Let $c$ and $d$ be the boxes participating in the corrective flip and suppose that $c \prec d$. If $c \notin b^{in}$ or $d \in b^{in}$, then $u_{b^{in}}^{D} = u_{b^{in}}^{D'}$. If $c \in b^{in}$ and $d \notin b^{in}$, then $u_{b^{in}}^{D} < u_{b^{in}}^{D'}$.
\end{proof}

In fact, a converse to Lemma \ref{lem:corrective_flips_decrease_length} holds as well.

\begin{proposition} \label{prop:only_corrective_flips_decrease_length}
Let $D$ and $D'$ be $\bullet/\circ/+$-diagrams with the same associated pair of permutations and suppose that $D'$ is obtained from $D$ by exchanging a single elbow piece and crossing tile in the associated pipe dreams. If $u_{b^{in}}^{D} \geq u_{b^{in}}^{D'}$ for all $b \in D$, then $D'$ was obtained from $D$ by performing a corrective flip.
\end{proposition}

\begin{proof}
If $D$ and $D'$ give the same pair of permutations and differ by exchanging a single elbow piece and crossing tile in their pipe dreams, the elbow piece and crossing tile exchanged must involve the same pair of pipes. Say the exchanged crossing tile is in box $c$ and the elbow piece is in box $d$. If $c$ contained a white stone in $D$, we must have $c \prec d$, otherwise $u_{b^{in}}^{D} < u_{b^{in}}^{D'}$ for any box $d \prec b \prec c$.  If there were some box $c \prec e \prec d$ which was an uncrossing pair with $c$ in $D$, then $u_{b^{in}}^{D} < u_{b^{in}}^{D'}$ for any box $e \prec b \prec d$. So, in this case $D'$ is obtained from $D$ via a corrective flip. If $c$ contained a black stone in $D$, we must have $d \prec c$, otherwise $u_{b^{in}}^{D} < u_{b^{in}}^{D'}$ for any box $c \prec b \prec d$. In this case, if $d$ had a crossing pair $e$ with $d \prec e \prec c$, then $u_{b^{in}}^{D} < u_{b^{in}}^{D'}$ for any box $c \prec b \prec e$.
\end{proof}

\begin{theorem} \label{thm:corrective_flips_terminate}
Every sequence of corrective flips terminates in a Go-diagram.
\end{theorem}

\begin{proof}
To a diagram $D$, we associate the tuple
\begin{displaymath}
\tau(D) = \bigoplus_{b \in D} u_{b^{in}}^{D} \in \bigoplus_{b \in D} \mathfrak{S}_n.
\end{displaymath}
\noindent
We endow $\bigoplus_{b \in D} \mathfrak{S}_n$ with the product partial order obtained from the Bruhat orders on each copy of $\mathfrak{S}_n$. Let $D'$ be obtained from $D$ by performing a corrective flip. Lemma \ref{lem:corrective_flips_decrease_length} implies $\tau(D') < \tau(D)$. So, any sequence of corrective flips must terminate. Then, point (ii) in Proposition \ref{prop:corrective_flip_observations} implies any sequence of corrective flips terminates in a Go-diagram.
\end{proof}

Unlike point (iii) in Theorem \ref{thm:positroids_via_le_moves}, there might be more than one Go-diagram obtainable from a $\bullet/\circ/+$-diagram via corrective flips. Consider the $\bullet/\circ/+$-diagram

\begin{equation} \label{eqn:boundary_example_after_flip}
\begin{tikzpicture}
\begin{scope}[scale=.75]
\draw[step = 1] (0,0) grid (3,3);
%
\draw[thick] (.5,.2) -- (.5,.8);
\draw[thick] (.2,.5) -- (.8,.5);
\draw[thick] (1.5,.2) -- (1.5,.8);
\draw[thick] (1.2,.5) -- (1.8,.5);
\draw[fill = white] (2.5,.5) circle (.25);
%
\draw[thick] (.5,1.2) -- (.5,1.8);
\draw[thick] (.2,1.5) -- (.8,1.5);
\draw[thick] (1.5,1.2) -- (1.5,1.8);
\draw[thick] (1.2,1.5) -- (1.8,1.5);
\draw[thick] (2.5,1.2) -- (2.5,1.8);
\draw[thick] (2.2,1.5) -- (2.8,1.5);
%
\draw[fill = black] (.5,2.5) circle (.25);
\draw[thick] (1.5,2.2) -- (1.5,2.8);
\draw[thick] (1.2,2.5) -- (1.8,2.5);
\draw[thick] (2.5,2.2) -- (2.5,2.8);
\draw[thick] (2.2,2.5) -- (2.8,2.5);
\draw (3.2,.1) node {,};
\end{scope}
\end{tikzpicture}
\end{equation}
\noindent
which is not a Go-diagram. Using corrective flips, it may be transformed into either
\begin{equation} \label{eqn:boundary_example_after_correcting}
\begin{tikzpicture}
\begin{scope}[scale=.75]
\draw[step = 1] (0,0) grid (3,3);
%
\draw[thick] (.5,.2) -- (.5,.8);
\draw[thick] (.2,.5) -- (.8,.5);
\draw[thick] (1.5,.2) -- (1.5,.8);
\draw[thick] (1.2,.5) -- (1.8,.5);
\draw[fill = white] (2.5,.5) circle (.25);
%
\draw[thick] (.5,1.2) -- (.5,1.8);
\draw[thick] (.2,1.5) -- (.8,1.5);
\draw[fill = black] (1.5,1.5) circle (.25);
\draw[thick] (2.5,1.2) -- (2.5,1.8);
\draw[thick] (2.2,1.5) -- (2.8,1.5);
%
\draw[thick] (.5,2.2) -- (.5,2.8);
\draw[thick] (.2,2.5) -- (.8,2.5);
\draw[thick] (1.5,2.2) -- (1.5,2.8);
\draw[thick] (1.2,2.5) -- (1.8,2.5);
\draw[thick] (2.5,2.2) -- (2.5,2.8);
\draw[thick] (2.2,2.5) -- (2.8,2.5);
\draw (4,1.5) node {or};
\begin{scope}[xshift=5cm]
\draw[step = 1] (0,0) grid (3,3);
%
\draw[thick] (.5,.2) -- (.5,.8);
\draw[thick] (.2,.5) -- (.8,.5);
\draw[thick] (1.5,.2) -- (1.5,.8);
\draw[thick] (1.2,.5) -- (1.8,.5);
\draw[thick] (2.5,.2) -- (2.5,.8);
\draw[thick] (2.2,.5) -- (2.8,.5);
%
\draw[thick] (.5,1.2) -- (.5,1.8);
\draw[thick] (.2,1.5) -- (.8,1.5);
\draw[fill = white] (1.5,1.5) circle (.25);
\draw[thick] (2.5,1.2) -- (2.5,1.8);
\draw[thick] (2.2,1.5) -- (2.8,1.5);
%
\draw[fill = black] (.5,2.5) circle (.25);
\draw[thick] (1.5,2.2) -- (1.5,2.8);
\draw[thick] (1.2,2.5) -- (1.8,2.5);
\draw[thick] (2.5,2.2) -- (2.5,2.8);
\draw[thick] (2.2,2.5) -- (2.8,2.5);
\draw (3.2,.1) node {.};
\end{scope}
\end{scope}
\end{tikzpicture}
\end{equation}

One could remove this aspect of free will from the definition of corrective flip, for instance by defining corrective flips to only switch pluses and white stones. However, we find Definition \ref{def:corrective_flip} is the correct choice of definition given Proposition \ref{prop:only_corrective_flips_decrease_length} and the role corrective flips play in the boundary structure of Deodhar components, described in Section \ref{sec:boundaries}.

\begin{remark}
The set of corrective flips as described is not a minimal set of moves with the properties described in Proposition \ref{prop:corrective_flip_observations} and Theorem \ref{thm:corrective_flips_terminate}. If one wanted a smaller set of moves with these properties, they could consider only corrective flips such that the elbow and crossing pieces being switched in the pipe dream have no other elbow pieces between them involving the same pair of pipes. However, even this set of moves isn't minimal: restricted to reduced diagrams, it is a strictly larger set of moves than the set of \Le-moves. It might be interesting to describe a set of corrective flips which is minimal and whose specialization to reduced diagrams is exactly the set of \Le-moves.
\end{remark}

\section{Boundaries} \label{sec:boundaries}

Throughout this section, stones in diagrams will be colored black if and only if they are uncrossings.

The main theorem of this section describes a particular instance of when there is a containment of closures of Deodhar components, $\overline{\calD'} \subset \overline{\calD}$, within a Schubert cell.

\begin{theorem} \label{thm:diagrammatic_boundaries}
Let $D$ and $D'$ be Go-diagrams with the same Ferrers shape indexing Deodhar components $\calD$ and $\calD'$. Then, $\overline{\calD'} \subset \overline{\calD}$ with $\dim(\calD) = \dim(\calD') + 1$ if $D$ is obtained by $D'$ by:
\begin{itemize}
\item[(i)] Choosing a crossing/uncrossing pair in $D'$,
\item[(ii)] replacing the two stones in this pair with pluses and relabelling the other stones in the diagram such that stone is colored black if and only if it is an uncrossing, then
\item[(iii)] performing corrective flips.
\end{itemize}
\noindent
or by:
\begin{itemize}
\item[(i)] Choosing a white stone without an uncrossing pair in $D'$ such that replacing this stone with a plus decreases the length of the diagram's permutation by exactly one,
\item[(ii)] replacing this white stone with a plus and relabelling the other stones in the diagram such that stone is colored black if and only if it is an uncrossing, then
\item[(iii)] performing corrective flips.
\end{itemize}
\end{theorem}

For an example of the first set of moves, consider the Go-diagram
\begin{displaymath}
\begin{tikzpicture}
\begin{scope}[scale=.75]
\draw (-1,1.5) node {$D' = $};
\draw[step = 1] (0,0) grid (3,3);
%
\draw[thick] (.5,.2) -- (.5,.8);
\draw[thick] (.2,.5) -- (.8,.5);
\draw[thick] (1.5,.2) -- (1.5,.8);
\draw[thick] (1.2,.5) -- (1.8,.5);
\draw[fill = white] (2.5,.5) circle (.25);
%
\draw[thick] (.5,1.2) -- (.5,1.8);
\draw[thick] (.2,1.5) -- (.8,1.5);
\draw[thick] (1.5,1.2) -- (1.5,1.8);
\draw[thick] (1.2,1.5) -- (1.8,1.5);
\draw[fill = white] (2.5,1.5) circle (.25);
%
\draw[fill = black] (.5,2.5) circle (.25);
\draw[fill = black] (1.5,2.5) circle (.25);
\draw[thick] (2.5,2.2) -- (2.5,2.8);
\draw[thick] (2.2,2.5) -- (2.8,2.5);
\draw (3.2,.1) node {.};
\end{scope}
\end{tikzpicture}
\end{displaymath}
\noindent
The white stone at $(2,4)$ and the black stone at $(1,5)$ form a crossing/uncrossing pair. Replacing the stones in these squares with pluses yields the diagram (\ref{eqn:boundary_example_after_flip}) from the previous section. We saw that this diagram could be transformed into either of the Go-diagrams (\ref{eqn:boundary_example_after_correcting}) via corrective flips. So, $\calD'$ is a codimension one boundary of both of the two Deodhar components labelled by the Go-diagrams (\ref{eqn:boundary_example_after_correcting}).

For an example of the second set of moves, consider the Go-diagram
\begin{displaymath}
\begin{tikzpicture}
\begin{scope}[scale=.75]
\draw (-1,1) node {$D' = $};
\draw[step = 1] (0,0) grid (4,2);
%
\draw[thick] (.5,.2) -- (.5,.8);
\draw[thick] (.2,.5) -- (.8,.5);
\draw[fill = white] (1.5,.5) circle (.25);
\draw[thick] (2.5,.2) -- (2.5,.8);
\draw[thick] (2.2,.5) -- (2.8,.5);
\draw[thick] (3.5,.2) -- (3.5,.8);
\draw[thick] (3.2,.5) -- (3.8,.5);
%
\draw[thick] (.5,1.2) -- (.5,1.8);
\draw[thick] (.2,1.5) -- (.8,1.5);
\draw[fill = white] (1.5,1.5) circle (.25);
\draw[fill = white] (2.5,1.5) circle (.25);
\draw[thick] (3.5,1.2) -- (3.5,1.8);
\draw[thick] (3.2,1.5) -- (3.8,1.5);
\draw (4.2,.1) node {.};
\end{scope}
\end{tikzpicture}
\end{displaymath}
\noindent
Replacing the white stone at $(1,5)$ with a plus decreases the length of the permutation by exactly one, so this replacement is valid. After, performing this replacement we obtain the diagram
\begin{displaymath}
\begin{tikzpicture}
\begin{scope}[scale=.75]
\draw[step = 1] (0,0) grid (4,2);
%
\draw[thick] (.5,.2) -- (.5,.8);
\draw[thick] (.2,.5) -- (.8,.5);
\draw[fill = white] (1.5,.5) circle (.25);
\draw[thick] (2.5,.2) -- (2.5,.8);
\draw[thick] (2.2,.5) -- (2.8,.5);
\draw[thick] (3.5,.2) -- (3.5,.8);
\draw[thick] (3.2,.5) -- (3.8,.5);
%
\draw[thick] (.5,1.2) -- (.5,1.8);
\draw[thick] (.2,1.5) -- (.8,1.5);
\draw[thick] (1.5,1.2) -- (1.5,1.8);
\draw[thick] (1.2,1.5) -- (1.8,1.5);
\draw[fill = white] (2.5,1.5) circle (.25);
\draw[thick] (3.5,1.2) -- (3.5,1.8);
\draw[thick] (3.2,1.5) -- (3.8,1.5);
\draw (4.2,.1) node {,};
\end{scope}
\end{tikzpicture}
\end{displaymath}
\noindent
which is not a Go-diagram. Performing corrective flips, which in this case are simply \Le-moves, we arrive at the diagram
\begin{displaymath}
\begin{tikzpicture}
\begin{scope}[scale=.75]
\draw (-1,1) node {$D = $};
\draw[step = 1] (0,0) grid (4,2);
%
\draw[thick] (.5,.2) -- (.5,.8);
\draw[thick] (.2,.5) -- (.8,.5);
\draw[thick] (1.5,.2) -- (1.5,.8);
\draw[thick] (1.2,.5) -- (1.8,.5);
\draw[thick] (2.5,.2) -- (2.5,.8);
\draw[thick] (2.2,.5) -- (2.8,.5);
\draw[thick] (3.5,.2) -- (3.5,.8);
\draw[thick] (3.2,.5) -- (3.8,.5);
%
\draw[fill = white] (.5,1.5) circle (.25);
\draw[thick] (1.5,1.2) -- (1.5,1.8);
\draw[thick] (1.2,1.5) -- (1.8,1.5);
\draw[fill = white] (2.5,1.5) circle (.25);
\draw[thick] (3.5,1.2) -- (3.5,1.8);
\draw[thick] (3.2,1.5) -- (3.8,1.5);
\draw (4.2,.1) node {.};
\end{scope}
\end{tikzpicture}
\end{displaymath}
\noindent
So, the Deodhar component $\calD'$ is a codimension one boundary of $\calD$. The diagrams $D$ and $D'$ are also \Le-diagrams and thus index positroid cells $\calP$ and $\calP'$. One also has that $\calP'$ is a codimension one boundary of $\calP$. To prove Theorem \ref{thm:diagrammatic_boundaries}, we will need a technical lemma.

\begin{lemma} \label{lem:technical_for_boundaries}
Let $D$ and $D'$ be $\bullet/\circ/+$-diagrams and suppose $D$ is obtained from $D'$ by performing a corrective flip. If $b$ contains a white stone in $D$, but not in $D'$, then $I_b^{D} > I_b^{D'}$. Moreover, let $J_b^{D'}$ be the set (\ref{eqn:J_b}) from Proposition \ref{prop:large_sets_vanish}. In the case above, $I_b^{D} > J_b^{D'}$.
\end{lemma}

\begin{proof}
Let $D$ be obtained from $D'$ by performing a corrective flip involving the pipes $i$ and $j$ with $i < j$. Let $b$ contain a white stone in $D$ but not in $D'$. From point (iv) in Proposition \ref{prop:corrective_flip_observations}, $b$ must be along either the pipe $i$ or $j$ between the two squares involved in the corrective flip. If $b$ contains a plus in $D'$, it must be the plus that was involved in the corrective flip. Then, $I_b^{D'} = I_b^{D} \setminus j \cup i$, so $I_b^{D} > I_b^{D'}$. In this case, $J_b^{D'} = I_b^{D'}$. So, $I_b^{D} > J_b^{D'}$ as well.

Suppose $b$ contains a black stone in $D'$. Then, the pipes coming into $b$ in pipe dream associated to $D'$ look like one of the four following cases.
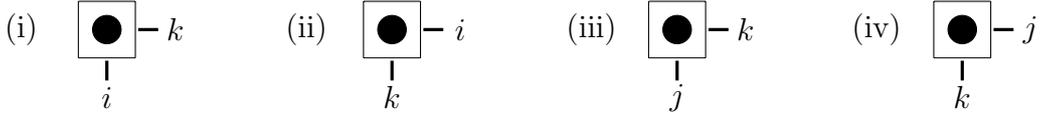
\begin{displaymath}
\begin{tikzpicture}
\begin{scope}[scale=.75]
\draw[step=1] (0,0) grid (1,1);
\draw[fill = black] (.5,.5) circle (.25);
\draw[very thick] (.5,-.4) -- (.5,-.05);
\draw (.5,-.7) node {$i$};
\draw[very thick] (1.05,.5) -- (1.4,.5);
\draw (1.7,.5) node {$k$};
\draw (-1,.5) node {(i)};

\begin{scope}[xshift = 5cm]
\draw[step=1] (0,0) grid (1,1);
\draw[fill = black] (.5,.5) circle (.25);
\draw[very thick] (.5,-.4) -- (.5,-.05);
\draw (.5,-.7) node {$k$};
\draw[very thick] (1.05,.5) -- (1.4,.5);
\draw (1.7,.5) node {$i$};
\draw (-1,.5) node {(ii)};
\end{scope}

\begin{scope}[xshift = 10cm]
\draw[step=1] (0,0) grid (1,1);
\draw[fill = black] (.5,.5) circle (.25);
\draw[very thick] (.5,-.4) -- (.5,-.05);
\draw (.5,-.7) node {$j$};
\draw[very thick] (1.05,.5) -- (1.4,.5);
\draw (1.7,.5) node {$k$};
\draw (-1,.5) node {(iii)};
\end{scope}

\begin{scope}[xshift = 15cm]
\draw[step=1] (0,0) grid (1,1);
\draw[fill = black] (.5,.5) circle (.25);
\draw[very thick] (.5,-.4) -- (.5,-.05);
\draw (.5,-.7) node {$k$};
\draw[very thick] (1.05,.5) -- (1.4,.5);
\draw (1.7,.5) node {$j$};
\draw (-1,.5) node {(iv)};
\end{scope}
\end{scope}
\end{tikzpicture}
\end{displaymath}

In case (ii), $k < i$. So, $k < j$ and $b$ would still be filled with a black stone when the corrective flip switches the pipes $i$ and $j$. Similarly, in case (iii), $j < k$ and $b$ would still be filled with a black stone in $D$.

In case (i), $b$ contains a white stone in $D$ if and only if $i < k < j$. The set $I_b^{D'}$ contains $i$ and $I_b^{D}$ contains $j$, since the square $b$ is temporarily filled with an elbow piece when determining $I_b$. So, $I_b^{D'} = I_b^{D} \setminus j \cup i$, and $I_b^{D} > I_b^{D'}$. In this case, $J_b^{D'} = I_b^{D} \setminus j \cup k$, and since $k < j$, $I_b^{D} > J_b^{D'}$.

Finally, in case (iv) note that since $b$ is along the segment of pipe between the two boxes in the corrective flip, the pipe $i$ appears somewhere below the box $b$ in its column. If $k = i$, then $b$ must be the uncrossing tile paired with the crossing which caused the corrective flip. In this case, $b$ will not be filled with a white stone after performing the corrective flip to reach $D$. Since the pipe $i$ appears somewhere below $b$ in its column and is not involved in the crossing at $b$, $i \in I_b^{D'}$. Since $b$ is temporarily filled with an elbow piece when determining $I_b^{D'}$, $j \notin I_b^{D'}$. So, again $I_b^{D'} = I_b^{D} \setminus j \cup i$, and $I_b^{D} > I_b^{D'}$. In this case, $J_b^{D'} = I_b^{D} \setminus j \cup k$, and since $i < k < j$, $I_b^{D} > J_b^{D'}$.
\end{proof}

\begin{proof}[Proof of Theorem \ref{thm:diagrammatic_boundaries}]
Let $D$ and $D'$ be Go-diagrams of shape $\lambda$ and suppose that $D$ is obtained from $D'$ by either of the two procedures in the theorem statement. Let $\calD$ and $\calD'$ be the Deodhard components indexed by $D$ and $D'$. Evidently,
\begin{displaymath}
\dim(\calD) = \dim(\calD') + 1,
\end{displaymath}
\noindent
so we need only check that $\overline{\calD'} \subset \overline{\calD}$. The ideal of the variety $\overline{\calD}$ is generated by $\Delta_b$, where $b$ contains a white stone in $D$, in addition to the defining relations for the Schubert variety $\calS_{v}$. Since the ideal of $\overline{\calD'}$ already contains the defining relations of $\calS_{v}$, it suffices to verify that $\Delta_b$ vanishes uniformly on $\calD'$ whenever $b$ contains a white stone in $D$.

Let $b = (i,j)$ contain a white stone in $D$. Let $D''$ be the intermediate $\bullet/\circ/+$-diagram obtained by undoing the crossing/uncrossing pair in $D'$, or by changing the appropriate white stone to a plus in $D'$, depending on how $D$ was obtained from $D'$, but without performing any corrective flips. Then,
\begin{displaymath}
u_{b^{in}}^{D'} \geq u_{b^{in}}^{D''}
\end{displaymath}
\noindent
for all $b \in \lambda$. This inequality is strict if $b^{in}$ contains the undone crossing tile but not the undone uncrossing (if it exists), and is an equality otherwise. Performing corrective flips to obtain $D$ from $D''$, Lemma \ref{lem:corrective_flips_decrease_length} implies $u_{b^{in}}^{D''} \geq u_{b^{in}}^{D}$. So, $u_{b^{in}}^{D'} \geq u_{b^{in}}^{D}$. Now, $I_b^{D}$ is the image of
\begin{displaymath}
u_{b^{in}}^{D} (u_b^D) (v^{D}_{b^{in}})^{-1} v^{D} (n,n-1, \dots, 2, 1)
\end{displaymath}
\noindent
in the quotient $\mathfrak{S}_n / (\mathfrak{S}_k \times \mathfrak{S}_{n-k})$. Since $D$ and $D'$ have the same Ferrers shape, $v^D = v^{D'}$ and $v^{D}_{b^{in}} = v^{D'}_{b^{in}}$. So, $I_b^{D'} \leq I_b^{D}$.

Suppose $b$ contains a white stone in $D'$. If $I_b^{D'} = I_b^{D}$, $\Delta_{I_b^{D}}$ clearly vanishes uniformly on $\calD$. Suppose that $I_b^{D'} < I_b^{D}$. Let $J_b^{D'}$ be the set (\ref{eqn:J_b}) from Proposition \ref{prop:large_sets_vanish}. Since $b$ contains a white stone, $J_b^{D'} < I_b^{D'}$. Then, Proposition \ref{prop:large_sets_vanish} implies that $\Delta_{I_b^{D}}$ vanishes uniformly on $\calD'$.

Suppose that $b$ contains a plus in $D'$. Then, $b$ must have either changed to a white stone when performing some corrective flip, or when replacing stones with pluses to obtain $D''$. If $b$ turned into a white stone when performing a corrective flip, Lemma \ref{lem:technical_for_boundaries} implies that $I_b^{D} > I_b^{D'}$. If $b$ turned into a white stone when replacing stones with pluses, an argument identical to the proof of Lemma \ref{lem:technical_for_boundaries} shows that $I_b^{D} > I_b^{D'}$. In either case, since $b$ is filled with a plus, $J_b^{D'} = I_b^{D'}$. So, Proposition \ref{prop:large_sets_vanish} implies that $\Delta_{I_b^{D}}$ vanishes uniformly on $\calD'$.

Finally, suppose that $b$ contains a black stone in $D'$. Again, $b$ must have either changed to a white stone when performing some corrective flip, or when replacing stones with pluses to obtain $D''$. If $b$ changed to white stone during a corrective flip, this occurred in either case (i) or (iv) examined in the proof of Lemma \ref{lem:technical_for_boundaries}, so $J_b^{D'} < I_b^{D}$ and Proposition \ref{prop:large_sets_vanish} implies that $\Delta_{I_b^{D}}$ vanishes uniformly on $\calD'$. The argument in the case where $b$ changed from being a black stone to a white stone when replacing stones with pluses to pass from $D'$ to $D''$ is similar.
\end{proof}

We conjecture that Theorem \ref{thm:diagrammatic_boundaries} is sharp for describing when there is a containment of closures of Deodhar components within a Schubert cell.

\begin{conjecture} \label{conj:main}
Let $D$ and $D'$ be Go-diagrams with the same Ferrers shape indexing Deodhar components $\calD$ and $\calD'$. Suppose that $\overline{\calD'} \subset \overline{\calD}$ and that $\dim(\calD) = \dim(\calD') +1$. Then, $D$ is obtained from $D'$ by one of the two procedures described in Theorem \ref{thm:diagrammatic_boundaries}.
\end{conjecture}

There is reason to be skeptical of this conjecture. Deodhar components are in general poorly behaved. In particular, Proposition 2.5 in \cite{dudas:not_a_stratification} shows that the closure of a Deodhar component is not in general a union of Deodhar components by providing two Deodhar components $\calD$ and $\calD'$ in the same Schubert cell in the type $B$ full flag manifold such that $\overline{\calD} \cap \calD'$ is a nonempty proper subset of $\calD'$. Proposition 2.7 in \cite{dudas:not_a_stratification} disproves a conjecture for determining whether $\overline{\calD} \cap \calD' \neq \emptyset$. While Conjecture \ref{conj:main} addresses a question distinct from these two issues, the general wild behavior of the Deodhar decomposition could be reason for skepticism. As a reality check we give an example that Conjecture \ref{conj:main} accurately describes the boundary structure for Deodhar components within a positroid cell, and Theorem \ref{thm:positroid_boundaries} verifies Conjecture \ref{conj:main} accurately describes the boundary structure of positroid varieties within a Schubert cell.

\begin{example}
The following poset consists of Go-diagrams labelling Deodhar components with the positroid cell $\calP_{123456,456123} \subset Gr(3,6)$ ordered by containment of closures of Deodhar components.

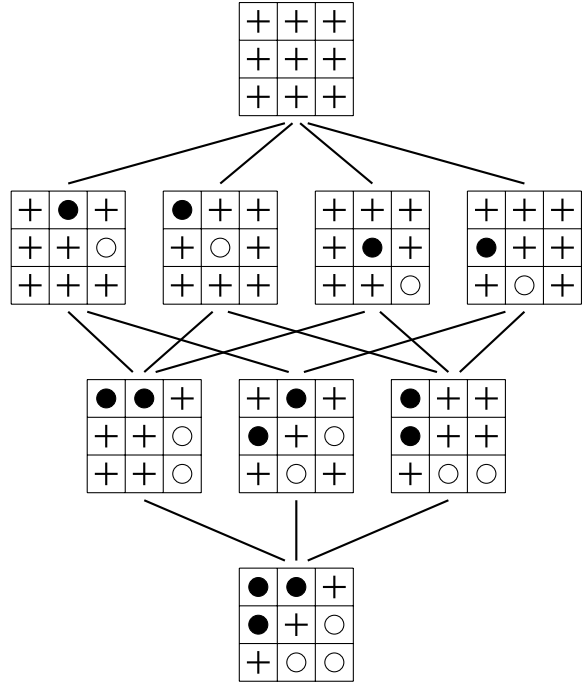
\begin{displaymath}
\begin{tikzpicture}
\begin{scope}[scale=.5]
\begin{scope}[shift={(0,-3)}]
\draw[step = 1] (0,0) grid (3,3);
%
\draw[thick] (.5,.2) -- (.5,.8);
\draw[thick] (.2,.5) -- (.8,.5);
\draw[fill = white] (1.5,.5) circle (.25);
\draw[fill = white] (2.5,.5) circle (.25);
%
\draw[fill = black] (.5,1.5) circle (.25);
\draw[thick] (1.5,1.2) -- (1.5,1.8);
\draw[thick] (1.2,1.5) -- (1.8,1.5);
\draw[fill = white] (2.5,1.5) circle (.25);
%
\draw[fill = black] (.5,2.5) circle (.25);
\draw[fill = black] (1.5,2.5) circle (.25);
\draw[thick] (2.5,2.2) -- (2.5,2.8);
\draw[thick] (2.2,2.5) -- (2.8,2.5);

\begin{scope}[shift={(-4,5)}]
\draw[step = 1] (0,0) grid (3,3);
%
\draw[thick] (.5,.2) -- (.5,.8);
\draw[thick] (.2,.5) -- (.8,.5);
\draw[thick] (1.5,.2) -- (1.5,.8);
\draw[thick] (1.2,.5) -- (1.8,.5);
\draw[fill = white] (2.5,.5) circle (.25);
%
\draw[thick] (.5,1.2) -- (.5,1.8);
\draw[thick] (.2,1.5) -- (.8,1.5);
\draw[thick] (1.5,1.2) -- (1.5,1.8);
\draw[thick] (1.2,1.5) -- (1.8,1.5);
\draw[fill = white] (2.5,1.5) circle (.25);
%
\draw[fill = black] (.5,2.5) circle (.25);
\draw[fill = black] (1.5,2.5) circle (.25);
\draw[thick] (2.5,2.2) -- (2.5,2.8);
\draw[thick] (2.2,2.5) -- (2.8,2.5);
\end{scope}

\begin{scope}[shift={(0,5)}]
\draw[step = 1] (0,0) grid (3,3);
%
\draw[thick] (.5,.2) -- (.5,.8);
\draw[thick] (.2,.5) -- (.8,.5);
\draw[fill = white] (1.5,.5) circle (.25);
\draw[thick] (2.5,.2) -- (2.5,.8);
\draw[thick] (2.2,.5) -- (2.8,.5);
%
\draw[fill = black] (.5,1.5) circle (.25);
\draw[thick] (1.5,1.2) -- (1.5,1.8);
\draw[thick] (1.2,1.5) -- (1.8,1.5);
\draw[fill = white] (2.5,1.5) circle (.25);
%
\draw[thick] (.5,2.2) -- (.5,2.8);
\draw[thick] (.2,2.5) -- (.8,2.5);
\draw[fill = black] (1.5,2.5) circle (.25);
\draw[thick] (2.5,2.2) -- (2.5,2.8);
\draw[thick] (2.2,2.5) -- (2.8,2.5);
\end{scope}

\begin{scope}[shift={(4,5)}]
\draw[step = 1] (0,0) grid (3,3);
%
\draw[thick] (.5,.2) -- (.5,.8);
\draw[thick] (.2,.5) -- (.8,.5);
\draw[fill = white] (1.5,.5) circle (.25);
\draw[fill = white] (2.5,.5) circle (.25);
%
\draw[fill = black] (.5,1.5) circle (.25);
\draw[thick] (1.5,1.2) -- (1.5,1.8);
\draw[thick] (1.2,1.5) -- (1.8,1.5);
\draw[thick] (2.5,1.2) -- (2.5,1.8);
\draw[thick] (2.2,1.5) -- (2.8,1.5);
%
\draw[fill = black] (.5,2.5) circle (.25);
\draw[thick] (1.5,2.2) -- (1.5,2.8);
\draw[thick] (1.2,2.5) -- (1.8,2.5);
\draw[thick] (2.5,2.2) -- (2.5,2.8);
\draw[thick] (2.2,2.5) -- (2.8,2.5);
\end{scope}

\begin{scope}[shift={(-6,10)}]
\draw[step = 1] (0,0) grid (3,3);
%
\draw[thick] (.5,.2) -- (.5,.8);
\draw[thick] (.2,.5) -- (.8,.5);
\draw[thick] (1.5,.2) -- (1.5,.8);
\draw[thick] (1.2,.5) -- (1.8,.5);
\draw[thick] (2.5,.2) -- (2.5,.8);
\draw[thick] (2.2,.5) -- (2.8,.5);
%
\draw[thick] (.5,1.2) -- (.5,1.8);
\draw[thick] (.2,1.5) -- (.8,1.5);
\draw[thick] (1.5,1.2) -- (1.5,1.8);
\draw[thick] (1.2,1.5) -- (1.8,1.5);
\draw[fill = white] (2.5,1.5) circle (.25);
%
\draw[thick] (.5,2.2) -- (.5,2.8);
\draw[thick] (.2,2.5) -- (.8,2.5);
\draw[fill = black] (1.5,2.5) circle (.25);
\draw[thick] (2.5,2.2) -- (2.5,2.8);
\draw[thick] (2.2,2.5) -- (2.8,2.5);
\end{scope}

\begin{scope}[shift={(-2,10)}]
\draw[step = 1] (0,0) grid (3,3);
%
\draw[thick] (.5,.2) -- (.5,.8);
\draw[thick] (.2,.5) -- (.8,.5);
\draw[thick] (1.5,.2) -- (1.5,.8);
\draw[thick] (1.2,.5) -- (1.8,.5);
\draw[thick] (2.5,.2) -- (2.5,.8);
\draw[thick] (2.2,.5) -- (2.8,.5);
%
\draw[thick] (.5,1.2) -- (.5,1.8);
\draw[thick] (.2,1.5) -- (.8,1.5);
\draw[fill = white] (1.5,1.5) circle (.25);
\draw[thick] (2.5,1.2) -- (2.5,1.8);
\draw[thick] (2.2,1.5) -- (2.8,1.5);
%
\draw[fill = black] (.5,2.5) circle (.25);
\draw[thick] (1.5,2.2) -- (1.5,2.8);
\draw[thick] (1.2,2.5) -- (1.8,2.5);
\draw[thick] (2.5,2.2) -- (2.5,2.8);
\draw[thick] (2.2,2.5) -- (2.8,2.5);
\end{scope}

\begin{scope}[shift={(2,10)}]
\draw[step = 1] (0,0) grid (3,3);
%
\draw[thick] (.5,.2) -- (.5,.8);
\draw[thick] (.2,.5) -- (.8,.5);
\draw[thick] (1.5,.2) -- (1.5,.8);
\draw[thick] (1.2,.5) -- (1.8,.5);
\draw[fill = white] (2.5,.5) circle (.25);
%
\draw[thick] (.5,1.2) -- (.5,1.8);
\draw[thick] (.2,1.5) -- (.8,1.5);
\draw[fill = black] (1.5,1.5) circle (.25);
\draw[thick] (2.5,1.2) -- (2.5,1.8);
\draw[thick] (2.2,1.5) -- (2.8,1.5);
%
\draw[thick] (.5,2.2) -- (.5,2.8);
\draw[thick] (.2,2.5) -- (.8,2.5);
\draw[thick] (1.5,2.2) -- (1.5,2.8);
\draw[thick] (1.2,2.5) -- (1.8,2.5);
\draw[thick] (2.5,2.2) -- (2.5,2.8);
\draw[thick] (2.2,2.5) -- (2.8,2.5);
\end{scope}

\begin{scope}[shift={(6,10)}]
\draw[step = 1] (0,0) grid (3,3);
%
\draw[thick] (.5,.2) -- (.5,.8);
\draw[thick] (.2,.5) -- (.8,.5);
\draw[fill = white] (1.5,.5) circle (.25);
\draw[thick] (2.5,.2) -- (2.5,.8);
\draw[thick] (2.2,.5) -- (2.8,.5);
%
\draw[fill = black] (.5,1.5) circle (.25);
\draw[thick] (1.5,1.2) -- (1.5,1.8);
\draw[thick] (1.2,1.5) -- (1.8,1.5);
\draw[thick] (2.5,1.2) -- (2.5,1.8);
\draw[thick] (2.2,1.5) -- (2.8,1.5);
%
\draw[thick] (.5,2.2) -- (.5,2.8);
\draw[thick] (.2,2.5) -- (.8,2.5);
\draw[thick] (1.5,2.2) -- (1.5,2.8);
\draw[thick] (1.2,2.5) -- (1.8,2.5);
\draw[thick] (2.5,2.2) -- (2.5,2.8);
\draw[thick] (2.2,2.5) -- (2.8,2.5);
\end{scope}

\begin{scope}[shift={(0,15)}]
\draw[step = 1] (0,0) grid (3,3);
%
\draw[thick] (.5,.2) -- (.5,.8);
\draw[thick] (.2,.5) -- (.8,.5);
\draw[thick] (1.5,.2) -- (1.5,.8);
\draw[thick] (1.2,.5) -- (1.8,.5);
\draw[thick] (2.5,.2) -- (2.5,.8);
\draw[thick] (2.2,.5) -- (2.8,.5);
%
\draw[thick] (.5,1.2) -- (.5,1.8);
\draw[thick] (.2,1.5) -- (.8,1.5);
\draw[thick] (1.5,1.2) -- (1.5,1.8);
\draw[thick] (1.2,1.5) -- (1.8,1.5);
\draw[thick] (2.5,1.2) -- (2.5,1.8);
\draw[thick] (2.2,1.5) -- (2.8,1.5);
%
\draw[thick] (.5,2.2) -- (.5,2.8);
\draw[thick] (.2,2.5) -- (.8,2.5);
\draw[thick] (1.5,2.2) -- (1.5,2.8);
\draw[thick] (1.2,2.5) -- (1.8,2.5);
\draw[thick] (2.5,2.2) -- (2.5,2.8);
\draw[thick] (2.2,2.5) -- (2.8,2.5);
\end{scope}
\end{scope}

\begin{scope}[yscale=-1]
\draw[thick] (1.2,-.2) -- (-2.5,-1.8);
\draw[thick] (1.5,-.2) -- (1.5,-1.8);
\draw[thick] (1.8,-.2) -- (5.5,-1.8);
\draw[thick] (-2.8,-5.2) -- (-4.5,-6.8);
\draw[thick] (-2.5,-5.2) -- (-.7,-6.8);
\draw[thick] (-2.2,-5.2) -- (3.3,-6.8);
\draw[thick] (1.3,-5.2) -- (-4,-6.8);
\draw[thick] (1.7,-5.2) -- (7,-6.8);
\draw[thick] (5.2,-5.2) -- (-.3,-6.8);
\draw[thick] (5.5,-5.2) -- (3.7,-6.8);
\draw[thick] (5.8,-5.2) -- (7.5,-6.8);
\draw[thick] (-4.5,-10.2) -- (1.2,-11.8);
\draw[thick] (-.5,-10.2) -- (1.4,-11.8);
\draw[thick] (3.5,-10.2) -- (1.6,-11.8);
\draw[thick] (7.5,-10.2) -- (1.8,-11.8);
\end{scope}
\end{scope}
\end{tikzpicture}
\end{displaymath}
\end{example}

\begin{theorem} \label{thm:positroid_boundaries}
Let $D$ and $D'$ be \Le-diagrams in the same Ferrers shape indexing positroid cells $\calP_{u,v}$ and $\calP_{u',v}$. Then, $\calP_{u',v}$ is a codimension one boundary of $\calP_{u,v}$ if and only $D$ is obtained by $D'$ by:
\begin{itemize}
\item[(i)] Choosing a white stone in $D'$ such that replacing this stone with a plus decreases the length of the diagram's permutation by exactly one,
\item[(ii)] replacing this white stone with a plus, then
\item[(iii)] performing \Le-moves.
\end{itemize}
\end{theorem}

\begin{proof}
Let $D$ and $D'$ be \Le-diagrams in the same Ferrers shape indexing positroid cells $\calP_{u,v}$ and $\calP_{u',v}$. Combining Theorem 5.10, Theorem 5.9, and Theorem 3.16 from \cite{knutson:juggling}, $\overline{\calP_{u',v}} \subset \overline{\calP_{u,v}}$ if and only if there is a containment of Bruhat intervals $[u',v] \subset [u,v]$. In this case, the codimension of $\overline{\calP_{u',v}}$ in $\overline{\calP_{u,v}}$ is $\ell(u) - \ell(u')$. Let ${\mathbf{u}}$ be the expression obtained by omitting the identity terms in the positive distinguished expression for $u$ in ${\mathbf{v}}$. Then, $\calP_{u',v}$ is a codimension one boundary of $\calP_{u,v}$ if and only if $\ell(u) - \ell(u')$ and there is an subexpression ${\mathbf{u'}}$ of ${\mathbf{u}}$ obtained by omitting one transposition of ${\mathbf{u}}$. This subexpression ${\mathbf{u'}}$ is diagrammatically realized by replacing the white stone corresponding to the omitted transposition in the \Le-diagram associated to $u$ with a plus. Theorem \ref{thm:positroids_via_le_moves} implies that the resulting diagram can be transformed into the \Le-diagram indexing $\calP_{u',v}$ using \Le-moves.
\end{proof}

\section{Classification of Go-diagrams} \label{sec:classification}

In this section, we do not assume that stones in diagrams are colored black if and only if they are uncrossings.

The goal of this section is to give a means of verifying whether an arbitrary filling of a Ferrers shape with black stones, white stones, and pluses is a Go-diagram. For \Le-diagrams, there is a compact description of the class of \Le-diagrams as diagrams avoiding certain subdiagrams. Theorem \ref{thm:sad_injection} shows that no reasonable description of Go-diagrams in terms of forbidden subdiagrams can exist. In lieu of such a description, Theorem \ref{thm:characterization} gives an inductive characterization of the class of Go-diagrams.

We say that a rectangular diagram is a {\textit{minimal violation}} if the only square in the diagram which violates the distinguished property is in the top left corner and the pair of pipes which should uncross in this square initially cross in the bottom right corner. Restricted to $\circ/+$-diagrams, diagrams which are minimal violations are of the form
\begin{equation} \label{eqn:le-violation}
\begin{tikzpicture}
\begin{scope}[scale=.65]
\draw (0,0) -- (0,5) -- (6,5) -- (6,0) -- (0,0);
\draw (1,0) -- (1,5);
\draw (0,4) -- (6,4);
\draw (5,5) -- (5,0);
\draw (6,1) -- (0,1);
\draw (0,2) -- (1,2);
\draw (0,3) -- (1,3);
\draw (2,5) -- (2,4);
\draw (4,5) -- (4,4);
\draw (5,3) -- (6,3);
\draw (5,2) -- (6,2);
\draw (4,0) -- (4,1);
\draw (2,0) -- (2,1);

\draw[thick] (.5,.2) -- (.5,.8);
\draw[thick] (.2,.5) -- (.8,.5);
%
\begin{scope}[shift={(0,4)}]
\draw (0,0) -- (1,1);
\draw[fill = white] (.3,.7) circle (.18);
\draw[thick] (.7,.1) -- (.7,.5);
\draw[thick] (.5,.3) -- (.9,.3);
\end{scope}
\draw[thick] (5.5,4.2) -- (5.5,4.8);
\draw[thick] (5.2,4.5) -- (5.8,4.5);
%
\draw[fill = white] (5.5,.5) circle (.25);
%
\draw[fill = white] (.5,1.5) circle (.25);
\draw[fill = white] (.5,3.5) circle (.25);
\draw[fill = white] (1.5,.5) circle (.25);
\draw[fill = white] (1.5,4.5) circle (.25);
\draw[fill = white] (4.5,4.5) circle (.25);
\draw[fill = white] (4.5,.5) circle (.25);
\draw[fill = white] (5.5,1.5) circle (.25);
\draw[fill = white] (5.5,3.5) circle (.25);
%
\draw (3,.5) node {$\dots$};
\draw (3,4.5) node {$\dots$};
\draw (.5,2.35) node {$.$};
\draw (.5,2.5) node {$.$};
\draw (.5,2.65) node {$.$};
\draw (5.5,2.35) node {$.$};
\draw (5.5,2.5) node {$.$};
\draw (5.5,2.65) node {$.$};
\draw (3,2.5) node {white stones};
\draw (6.2,0) node {$.$};
\end{scope}
\end{tikzpicture}
\end{equation}
The set of Go-diagrams which are $\circ/+$-diagrams is exactly the set of \Le-diagrams. Any diagram which is not a \Le-diagram contains a minimal violation as a subdiagram. The following theorem shows that the set of minimal violations for $\bullet/\circ/+$-diagrams is much more poorly behaved by providing an injection from the set of Go-diagrams into the set of minimal violations. Since every minimal violation must appear on any list of forbidden subdiagrams for the class of Go-diagrams, this shows that Go-diagrams do not admit a reasonable description in terms of forbidden subdiagrams. This provides a negative answer to Problem 4.9 in \cite{kodama:deodhar_decomposition}.

\begin{theorem} \label{thm:sad_injection}
There is an injection from the set of valid Go-diagrams into the set of minimal violations for the class of Go-diagrams.
\end{theorem}

\begin{proof}
Let $D$ be a Go-diagram of shape $\lambda = (\lambda_1, \lambda_2, \dots, \lambda_{\ell})$. Let $D'$ be the diagram of inside a $\lambda_{\ell} \times \ell$ rectangle obtained by placing $D$ in the top left corner, then padding out the bottom right corner with pluses. Note that $D'$ is a Go-diagram.

Consider the the $2 \times 2$ diagram
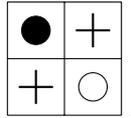
\begin{equation} \label{eqn:2_by_2_widget}
\begin{tikzpicture}
\begin{scope}[scale = .75]
\draw[step = 1] (0,0) grid (2,2);
%
\draw[thick] (.5,.2) -- (0.5,.8);
\draw[thick] (.2,.5) -- (.8,.5);
\draw[fill = white] (1.5,.5) circle (.25);
%
\draw[thick] (1.5,1.2) -- (1.5,1.8);
\draw[thick] (1.2,1.5) -- (1.8,1.5);
\draw[fill = black] (.5,1.5) circle (.25);
\draw (2.2,0) node {$.$};
\end{scope}
\end{tikzpicture}
\end{equation}
\noindent
We build the reflection of the shape $\lambda$ over the line $y = x$ using these $2 \times 2$ blocks. Call this figure $\Lambda$. Now build a rectangular diagram $D''$ which contains $D'$ in the top left corner, $\Lambda$ in the bottom right corner, and pluses padding out the rest of the squares. The dimensions of $D''$ do not matter as long as there is enough room that no square in $D'$ is adjacent to a square in $\Lambda$. Observe that $D''$ is a valid Go-diagram.

Finally, build a one box wide border around $D''$ which has:
\begin{itemize}
\item pluses in the top left, top right, and bottom left corners,
\item a white stone in the bottom right corner,
\item white stones along the bottom and right sides,
\item white or black stones along the top and left sides, as is necessary to avoid a violation of the distinguished property.
\end{itemize}
\noindent
In this diagram, the top left square should be an uncrossing with the bottom right square and hence violates the distinguished property. As no other square in the diagram violates the distinguished property, this diagram is a minimal violation.

Given a diagram of this form, one can recover the diagram $D$ it came from. To do so, first delete a one square wide strip of boxes from the boundary of the diagram. Then, examine the bottom right portion of this diagram to find a Ferrers shape built out of copies of the $2 \times 2$ subdiagram (\ref{eqn:2_by_2_widget}). The boxes of this same Ferrers shape in the top left corner are the diagram $D$. Since this map is reversible, it is an injection from the set of Go-diagrams to the set of minimal violations for the class of Go-diagrams.
\end{proof}

\begin{example}
Consider the Go-diagram
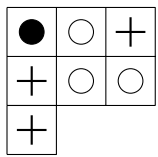
\begin{displaymath}
\begin{tikzpicture}
\begin{scope}[scale = .65]
\draw[step = 1] (0,0) grid (3,2);
\draw[step = 1] (0,-1) grid (1,0);
%
%
\draw[thick] (.5,-.2) -- (0.5,-.8);
\draw[thick] (.2,-.5) -- (.8,-.5);
\draw[thick] (.5,.2) -- (0.5,.8);
\draw[thick] (.2,.5) -- (.8,.5);
\draw[fill = white] (1.5,.5) circle (.25);
\draw[fill = white] (2.5,.5) circle (.25);
%
\draw[fill = black] (.5,1.5) circle (.25);
\draw[fill = white] (1.5,1.5) circle (.25);
\draw[thick] (2.5,1.2) -- (2.5,1.8);
\draw[thick] (2.2,1.5) -- (2.8,1.5);
\draw (3.2,-1) node {$.$};
\end{scope}
\end{tikzpicture}
\end{displaymath}
\noindent
The image of this Go-diagram under the injection described in Theorem \ref{thm:sad_injection} is
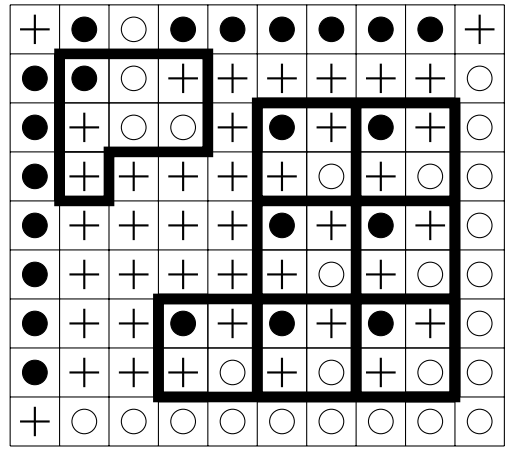
\begin{displaymath}
\begin{tikzpicture}
\begin{scope}[scale = .65]
\draw[step = 1] (0,0) grid (10,9);
\begin{scope}[shift={(1,6)}]
\draw[line width = 4] (0,-1) -- (0,2) -- (3,2) -- (3,0) -- (1,0) -- (1,-1) -- (0,-1) -- (0,2);
\draw[step = 1] (0,0) grid (3,2);
\draw[step = 1] (0,-1) grid (1,0);
%
%
\draw[thick] (.5,-.2) -- (0.5,-.8);
\draw[thick] (.2,-.5) -- (.8,-.5);
\draw[thick] (.5,.2) -- (0.5,.8);
\draw[thick] (.2,.5) -- (.8,.5);
\draw[fill = white] (1.5,.5) circle (.25);
\draw[fill = white] (2.5,.5) circle (.25);
%
\draw[fill = black] (.5,1.5) circle (.25);
\draw[fill = white] (1.5,1.5) circle (.25);
\draw[thick] (2.5,1.2) -- (2.5,1.8);
\draw[thick] (2.2,1.5) -- (2.8,1.5);
\end{scope}
\begin{scope}[shift={(3,1)}]
\draw[line width = 4, step = 2] (0,0) grid (6,2);
\draw[line width = 4, step = 2] (2,0) grid (6,6);
\draw[line width = 4] (1,0) -- (0,0) -- (0,2) -- (1,2);
\draw[line width = 4] (2,5) -- (2,6) -- (6,6) -- (6,0) -- (5,0);
\draw[step = 1] (0,0) grid (2,2);
\draw[thick] (.5,.2) -- (0.5,.8);
\draw[thick] (.2,.5) -- (.8,.5);
\draw[fill = white] (1.5,.5) circle (.25);
%
\draw[thick] (1.5,1.2) -- (1.5,1.8);
\draw[thick] (1.2,1.5) -- (1.8,1.5);
\draw[fill = black] (.5,1.5) circle (.25);
\begin{scope}[shift={(2,0)}]
\draw[step = 1] (0,0) grid (2,2);
\draw[thick] (.5,.2) -- (0.5,.8);
\draw[thick] (.2,.5) -- (.8,.5);
\draw[fill = white] (1.5,.5) circle (.25);
%
\draw[thick] (1.5,1.2) -- (1.5,1.8);
\draw[thick] (1.2,1.5) -- (1.8,1.5);
\draw[fill = black] (.5,1.5) circle (.25);
\end{scope}
\begin{scope}[shift={(4,0)}]
\draw[step = 1] (0,0) grid (2,2);
\draw[thick] (.5,.2) -- (0.5,.8);
\draw[thick] (.2,.5) -- (.8,.5);
\draw[fill = white] (1.5,.5) circle (.25);
%
\draw[thick] (1.5,1.2) -- (1.5,1.8);
\draw[thick] (1.2,1.5) -- (1.8,1.5);
\draw[fill = black] (.5,1.5) circle (.25);
\end{scope}
\begin{scope}[shift={(2,2)}]
\draw[step = 1] (0,0) grid (2,2);
\draw[thick] (.5,.2) -- (0.5,.8);
\draw[thick] (.2,.5) -- (.8,.5);
\draw[fill = white] (1.5,.5) circle (.25);
%
\draw[thick] (1.5,1.2) -- (1.5,1.8);
\draw[thick] (1.2,1.5) -- (1.8,1.5);
\draw[fill = black] (.5,1.5) circle (.25);
\end{scope}
\begin{scope}[shift={(4,2)}]
\draw[step = 1] (0,0) grid (2,2);
\draw[thick] (.5,.2) -- (0.5,.8);
\draw[thick] (.2,.5) -- (.8,.5);
\draw[fill = white] (1.5,.5) circle (.25);
%
\draw[thick] (1.5,1.2) -- (1.5,1.8);
\draw[thick] (1.2,1.5) -- (1.8,1.5);
\draw[fill = black] (.5,1.5) circle (.25);
\end{scope}
\begin{scope}[shift={(2,4)}]
\draw[step = 1] (0,0) grid (2,2);
\draw[thick] (.5,.2) -- (0.5,.8);
\draw[thick] (.2,.5) -- (.8,.5);
\draw[fill = white] (1.5,.5) circle (.25);
%
\draw[thick] (1.5,1.2) -- (1.5,1.8);
\draw[thick] (1.2,1.5) -- (1.8,1.5);
\draw[fill = black] (.5,1.5) circle (.25);
\end{scope}
\begin{scope}[shift={(4,4)}]
\draw[step = 1] (0,0) grid (2,2);
\draw[thick] (.5,.2) -- (0.5,.8);
\draw[thick] (.2,.5) -- (.8,.5);
\draw[fill = white] (1.5,.5) circle (.25);
%
\draw[thick] (1.5,1.2) -- (1.5,1.8);
\draw[thick] (1.2,1.5) -- (1.8,1.5);
\draw[fill = black] (.5,1.5) circle (.25);
\end{scope}
\end{scope}
\draw[thick] (.5,.2) -- (0.5,.8); \draw[thick] (.2,.5) -- (.8,.5);
\begin{scope}[shift={(1,1)}] \draw[thick] (.5,.2) -- (0.5,.8); \draw[thick] (.2,.5) -- (.8,.5); \end{scope}
\begin{scope}[shift={(2,1)}] \draw[thick] (.5,.2) -- (0.5,.8); \draw[thick] (.2,.5) -- (.8,.5); \end{scope}
\begin{scope}[shift={(1,2)}] \draw[thick] (.5,.2) -- (0.5,.8); \draw[thick] (.2,.5) -- (.8,.5); \end{scope}
\begin{scope}[shift={(2,2)}] \draw[thick] (.5,.2) -- (0.5,.8); \draw[thick] (.2,.5) -- (.8,.5); \end{scope}
\begin{scope}[shift={(1,3)}] \draw[thick] (.5,.2) -- (0.5,.8); \draw[thick] (.2,.5) -- (.8,.5); \end{scope}
\begin{scope}[shift={(2,3)}] \draw[thick] (.5,.2) -- (0.5,.8); \draw[thick] (.2,.5) -- (.8,.5); \end{scope}
\begin{scope}[shift={(3,3)}] \draw[thick] (.5,.2) -- (0.5,.8); \draw[thick] (.2,.5) -- (.8,.5); \end{scope}
\begin{scope}[shift={(4,3)}] \draw[thick] (.5,.2) -- (0.5,.8); \draw[thick] (.2,.5) -- (.8,.5); \end{scope}
\begin{scope}[shift={(1,4)}] \draw[thick] (.5,.2) -- (0.5,.8); \draw[thick] (.2,.5) -- (.8,.5); \end{scope}
\begin{scope}[shift={(2,4)}] \draw[thick] (.5,.2) -- (0.5,.8); \draw[thick] (.2,.5) -- (.8,.5); \end{scope}
\begin{scope}[shift={(3,4)}] \draw[thick] (.5,.2) -- (0.5,.8); \draw[thick] (.2,.5) -- (.8,.5); \end{scope}
\begin{scope}[shift={(4,4)}] \draw[thick] (.5,.2) -- (0.5,.8); \draw[thick] (.2,.5) -- (.8,.5); \end{scope}
\begin{scope}[shift={(2,5)}] \draw[thick] (.5,.2) -- (0.5,.8); \draw[thick] (.2,.5) -- (.8,.5); \end{scope}
\begin{scope}[shift={(3,5)}] \draw[thick] (.5,.2) -- (0.5,.8); \draw[thick] (.2,.5) -- (.8,.5); \end{scope}
\begin{scope}[shift={(4,5)}] \draw[thick] (.5,.2) -- (0.5,.8); \draw[thick] (.2,.5) -- (.8,.5); \end{scope}
\begin{scope}[shift={(4,6)}] \draw[thick] (.5,.2) -- (0.5,.8); \draw[thick] (.2,.5) -- (.8,.5); \end{scope}
\begin{scope}[shift={(4,7)}] \draw[thick] (.5,.2) -- (0.5,.8); \draw[thick] (.2,.5) -- (.8,.5); \end{scope}
\begin{scope}[shift={(5,7)}] \draw[thick] (.5,.2) -- (0.5,.8); \draw[thick] (.2,.5) -- (.8,.5); \end{scope}
\begin{scope}[shift={(6,7)}] \draw[thick] (.5,.2) -- (0.5,.8); \draw[thick] (.2,.5) -- (.8,.5); \end{scope}
\begin{scope}[shift={(7,7)}] \draw[thick] (.5,.2) -- (0.5,.8); \draw[thick] (.2,.5) -- (.8,.5); \end{scope}
\begin{scope}[shift={(8,7)}] \draw[thick] (.5,.2) -- (0.5,.8); \draw[thick] (.2,.5) -- (.8,.5); \end{scope}
\begin{scope}[shift={(0,8)}] \draw[thick] (.5,.2) -- (0.5,.8); \draw[thick] (.2,.5) -- (.8,.5); \end{scope}
\begin{scope}[shift={(9,8)}] \draw[thick] (.5,.2) -- (0.5,.8); \draw[thick] (.2,.5) -- (.8,.5); \end{scope}
\begin{scope}[shift={(1,0)}] \draw[fill = white] (.5,.5) circle (.25); \end{scope}
\begin{scope}[shift={(2,0)}] \draw[fill = white] (.5,.5) circle (.25); \end{scope}
\begin{scope}[shift={(3,0)}] \draw[fill = white] (.5,.5) circle (.25); \end{scope}
\begin{scope}[shift={(4,0)}] \draw[fill = white] (.5,.5) circle (.25); \end{scope}
\begin{scope}[shift={(5,0)}] \draw[fill = white] (.5,.5) circle (.25); \end{scope}
\begin{scope}[shift={(6,0)}] \draw[fill = white] (.5,.5) circle (.25); \end{scope}
\begin{scope}[shift={(7,0)}] \draw[fill = white] (.5,.5) circle (.25); \end{scope}
\begin{scope}[shift={(8,0)}] \draw[fill = white] (.5,.5) circle (.25); \end{scope}
\begin{scope}[shift={(9,0)}] \draw[fill = white] (.5,.5) circle (.25); \end{scope}
\begin{scope}[shift={(9,1)}] \draw[fill = white] (.5,.5) circle (.25); \end{scope}
\begin{scope}[shift={(9,2)}] \draw[fill = white] (.5,.5) circle (.25); \end{scope}
\begin{scope}[shift={(9,3)}] \draw[fill = white] (.5,.5) circle (.25); \end{scope}
\begin{scope}[shift={(9,4)}] \draw[fill = white] (.5,.5) circle (.25); \end{scope}
\begin{scope}[shift={(9,5)}] \draw[fill = white] (.5,.5) circle (.25); \end{scope}
\begin{scope}[shift={(9,6)}] \draw[fill = white] (.5,.5) circle (.25); \end{scope}
\begin{scope}[shift={(9,7)}] \draw[fill = white] (.5,.5) circle (.25); \end{scope}
\begin{scope}[shift={(0,1)}] \draw[fill = black] (.5,.5) circle (.25); \end{scope}
\begin{scope}[shift={(0,2)}] \draw[fill = black] (.5,.5) circle (.25); \end{scope}
\begin{scope}[shift={(0,3)}] \draw[fill = black] (.5,.5) circle (.25); \end{scope}
\begin{scope}[shift={(0,4)}] \draw[fill = black] (.5,.5) circle (.25); \end{scope}
\begin{scope}[shift={(0,5)}] \draw[fill = black] (.5,.5) circle (.25); \end{scope}
\begin{scope}[shift={(0,6)}] \draw[fill = black] (.5,.5) circle (.25); \end{scope}
\begin{scope}[shift={(0,7)}] \draw[fill = black] (.5,.5) circle (.25); \end{scope}
\begin{scope}[shift={(1,8)}] \draw[fill = black] (.5,.5) circle (.25); \end{scope}
\begin{scope}[shift={(2,8)}] \draw[fill = white] (.5,.5) circle (.25); \end{scope}
\begin{scope}[shift={(3,8)}] \draw[fill = black] (.5,.5) circle (.25); \end{scope}
\begin{scope}[shift={(4,8)}] \draw[fill = black] (.5,.5) circle (.25); \end{scope}
\begin{scope}[shift={(5,8)}] \draw[fill = black] (.5,.5) circle (.25); \end{scope}
\begin{scope}[shift={(6,8)}] \draw[fill = black] (.5,.5) circle (.25); \end{scope}
\begin{scope}[shift={(7,8)}] \draw[fill = black] (.5,.5) circle (.25); \end{scope}
\begin{scope}[shift={(8,8)}] \draw[fill = black] (.5,.5) circle (.25); \end{scope}
\draw (10.2,0) node {$.$};
\end{scope}
\end{tikzpicture}
\end{displaymath}
\end{example}

All the minimal violations obtained via the injection of Theorem \ref{thm:sad_injection} have the feature that, in the associated pipe dream, the pipes involved in the violation of the distinguished property take only one turn each. However, this is not true in general; it is possible for the pipes involved in the violation of the distinguished property to take arbitrarily many turns. For instance, one can arrange copies of the $2 \times 2$ block (\ref{eqn:2_by_2_widget}) in a serpentine pattern. 
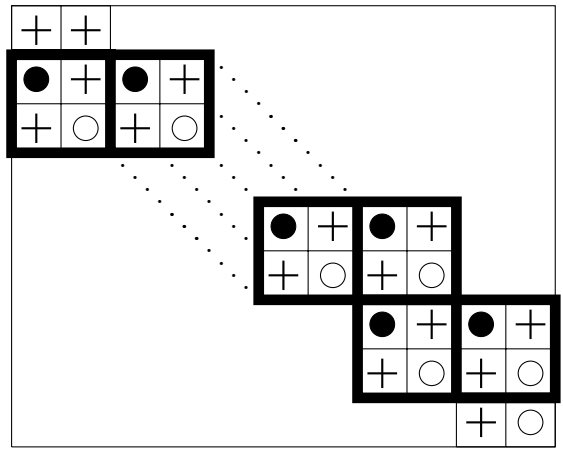
\begin{displaymath}
\begin{tikzpicture}
\begin{scope}[scale = .65]
\draw (0,0) -- (11,0) -- (11,9) -- (0,9) -- (0,0);
\draw[step=1] (11,0) grid (9,3);
\draw[step=1] (9,1) grid (7,5);
\draw[step=1] (7,3) grid (5,5);
\draw[line width = 4pt] (11,1) -- (7,1) -- (7,5);
\draw[line width = 4pt] (9,1) -- (9,5) -- (5,5) -- (5,3) -- (9,3);
\draw[line width = 4pt] (9,1) -- (11,1) -- (11,3) -- (7,3);
\draw[step=1] (0,9) grid (2,6);
\draw[step=1] (2,8) grid (4,6);
\draw[line width = 4pt] (0,8) -- (4,8) -- (4,6) -- (0,6) -- (0,8) -- (2,8) -- (2,6);
%
%
%
\begin{scope}[xshift = 9cm, yshift = 1cm]
\draw[thick] (.5,.2) -- (0.5,.8);
\draw[thick] (.2,.5) -- (.8,.5);
\draw[fill = white] (1.5,.5) circle (.25);
%
\draw[thick] (1.5,1.2) -- (1.5,1.8);
\draw[thick] (1.2,1.5) -- (1.8,1.5);
\draw[fill = black] (.5,1.5) circle (.25);
\end{scope}
\begin{scope}[xshift = 7cm, yshift = 1cm]
\draw[thick] (.5,.2) -- (0.5,.8);
\draw[thick] (.2,.5) -- (.8,.5);
\draw[fill = white] (1.5,.5) circle (.25);
%
\draw[thick] (1.5,1.2) -- (1.5,1.8);
\draw[thick] (1.2,1.5) -- (1.8,1.5);
\draw[fill = black] (.5,1.5) circle (.25);
\end{scope}
\begin{scope}[xshift = 7cm, yshift = 3cm]
\draw[thick] (.5,.2) -- (0.5,.8);
\draw[thick] (.2,.5) -- (.8,.5);
\draw[fill = white] (1.5,.5) circle (.25);
%
\draw[thick] (1.5,1.2) -- (1.5,1.8);
\draw[thick] (1.2,1.5) -- (1.8,1.5);
\draw[fill = black] (.5,1.5) circle (.25);
\end{scope}
\begin{scope}[xshift = 5cm, yshift = 3cm]
\draw[thick] (.5,.2) -- (0.5,.8);
\draw[thick] (.2,.5) -- (.8,.5);
\draw[fill = white] (1.5,.5) circle (.25);
%
\draw[thick] (1.5,1.2) -- (1.5,1.8);
\draw[thick] (1.2,1.5) -- (1.8,1.5);
\draw[fill = black] (.5,1.5) circle (.25);
\end{scope}
\begin{scope}[xshift = 0cm, yshift = 6cm]
\draw[thick] (.5,.2) -- (0.5,.8);
\draw[thick] (.2,.5) -- (.8,.5);
\draw[fill = white] (1.5,.5) circle (.25);
%
\draw[thick] (1.5,1.2) -- (1.5,1.8);
\draw[thick] (1.2,1.5) -- (1.8,1.5);
\draw[fill = black] (.5,1.5) circle (.25);
\end{scope}
\begin{scope}[xshift = 2cm, yshift = 6cm]
\draw[thick] (.5,.2) -- (0.5,.8);
\draw[thick] (.2,.5) -- (.8,.5);
\draw[fill = white] (1.5,.5) circle (.25);
%
\draw[thick] (1.5,1.2) -- (1.5,1.8);
\draw[thick] (1.2,1.5) -- (1.8,1.5);
\draw[fill = black] (.5,1.5) circle (.25);
\end{scope}
%
%
\draw[thick] (9.5,.2) -- (9.5,.8);
\draw[thick] (9.2,.5) -- (9.8,.5);
\draw[fill = white] (10.5,.5) circle (.25);
%
%
\draw[thick] (.5,8.2) -- (.5,8.8);
\draw[thick] (.2,8.5) -- (.8,8.5);
\draw[thick] (1.5,8.2) -- (1.5,8.8);
\draw[thick] (1.2,8.5) -- (1.8,8.5);
%
%
\draw (4.25,5.75) node {$.$};
\draw (4.5,5.5) node {$.$};
\draw (4.75,5.25) node {$.$};
\begin{scope}[xshift=1cm, yshift=0cm]
\draw (4.25,5.75) node {$.$};
\draw (4.5,5.5) node {$.$};
\draw (4.75,5.25) node {$.$};
\end{scope}
\draw (5,6) node {$.$};
\begin{scope}[xshift=0cm, yshift=1cm]
\draw (4.25,5.75) node {$.$};
\draw (4.5,5.5) node {$.$};
\draw (4.75,5.25) node {$.$};
\end{scope}
\begin{scope}[xshift = -1cm, yshift = -1cm]
\begin{scope}[xshift=1cm, yshift=0cm]
\draw (4.25,5.75) node {$.$};
\draw (4.5,5.5) node {$.$};
\draw (4.75,5.25) node {$.$};
\end{scope}
\draw (5,6) node {$.$};
\begin{scope}[xshift=0cm, yshift=1cm]
\draw (4.25,5.75) node {$.$};
\draw (4.5,5.5) node {$.$};
\draw (4.75,5.25) node {$.$};
\end{scope}
\end{scope}
\begin{scope}[xshift = -1cm, yshift = -2cm]
\begin{scope}[xshift=1cm, yshift=0cm]
\draw (4.25,5.75) node {$.$};
\draw (4.5,5.5) node {$.$};
\draw (4.75,5.25) node {$.$};
\end{scope}
\draw (5,6) node {$.$};
\begin{scope}[xshift=0cm, yshift=1cm]
\draw (4.25,5.75) node {$.$};
\draw (4.5,5.5) node {$.$};
\draw (4.75,5.25) node {$.$};
\end{scope}
\end{scope}
\begin{scope}[xshift = -2cm, yshift = -1cm]
\begin{scope}[xshift=1cm, yshift=0cm]
\draw (4.25,5.75) node {$.$};
\draw (4.5,5.5) node {$.$};
\draw (4.75,5.25) node {$.$};
\end{scope}
\draw (5,6) node {$.$};
\begin{scope}[xshift=0cm, yshift=1cm]
\draw (4.25,5.75) node {$.$};
\draw (4.5,5.5) node {$.$};
\draw (4.75,5.25) node {$.$};
\end{scope}
\end{scope}
\begin{scope}[xshift = 2cm, yshift=2cm]
\begin{scope}[xshift = -1cm, yshift = -2cm]
\begin{scope}[xshift=1cm, yshift=0cm]
\draw (4.25,5.75) node {$.$};
\draw (4.5,5.5) node {$.$};
\draw (4.75,5.25) node {$.$};
\end{scope}
\draw (5,6) node {$.$};
\begin{scope}[xshift=0cm, yshift=1cm]
\draw (4.25,5.75) node {$.$};
\draw (4.5,5.5) node {$.$};
\draw (4.75,5.25) node {$.$};
\end{scope}
\end{scope}
\begin{scope}[xshift = -2cm, yshift = -1cm]
\begin{scope}[xshift=1cm, yshift=0cm]
\draw (4.25,5.75) node {$.$};
\draw (4.5,5.5) node {$.$};
\draw (4.75,5.25) node {$.$};
\end{scope}
\draw (5,6) node {$.$};
\begin{scope}[xshift=0cm, yshift=1cm]
\draw (4.25,5.75) node {$.$};
\draw (4.5,5.5) node {$.$};
\draw (4.75,5.25) node {$.$};
\end{scope}
\end{scope}
\end{scope}
\end{scope}
\end{tikzpicture}
\end{displaymath}
\noindent
Here, we've highlighted to subdiagrams (\ref{eqn:2_by_2_widget}) to make the pattern clearer. All boxes not drawn are filled with white stones. One may check that this example is a minimal violation.

In lieu of a good description of Go-diagrams in terms of forbidden subdiagrams, we offer an algorithmic characterization of when a filling of a Ferrers shape with black stones, white stones, and pluses is a Go-diagram. Algorithm \ref{alg:partner} provides a method of producing a {\textit{partner}} square to any square in the diagram. A $\bullet/\circ/+$-diagram will be a Go-diagram if and only if a square has a partner if and only if it's filled with a black stone. In general, the partner of a black stone will be different than the white stone it serves as an uncrossing pair to. This notion of partner has two advantages over crossing/uncrossing pairs:
\begin{itemize}
\item Replacing all black stones and their partners with pluses simultaneously yields a reduced $\circ/+$-diagram for the same pair of permutations.
\item For a black stone in box $b$, replacing all black stones in $b^{in}$ and their partners with pluses simultaneously does not alter the location of $b$'s partner.
\end{itemize}
The following example shows that these properties are not enjoyed by crossing/uncrossing pairs.

\begin{example}
Consider the following Go-diagram, where the boxes containing one crossing/uncrossing pair have been shaded blue (dark gray in grayscale) and those containing the other have been shaded yellow (light gray).
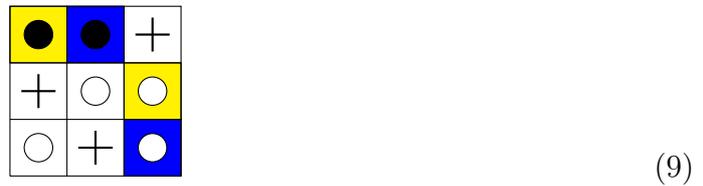
\begin{equation} \label{eqn:shaded_box_example}
\begin{tikzpicture}
\begin{scope}[scale = .75]
\draw [fill=yellow] (0,3) rectangle (1,2);
\draw [fill=yellow] (2,2) rectangle (3,1);
\draw [fill=blue] (1,3) rectangle (2,2);
\draw [fill=blue] (2,1) rectangle (3,0);
\draw (0,0) -- (0,3) -- (3,3) -- (3,0) -- (0,0);
\draw (0,1) -- (3,1);
\draw (0,2) -- (3,2);
\draw (1,0) -- (1,3);
\draw (2,0) -- (2,3);
\draw[fill = black] (.5,2.5) circle (.25);
\draw[fill = black] (1.5,2.5) circle (.25);
\draw[thick] (2.5,2.2) -- (2.5,2.8);
\draw[thick] (2.2,2.5) -- (2.8,2.5);
\draw[thick] (.5,1.2) -- (0.5,1.8);
\draw[thick] (.2,1.5) -- (.8,1.5);
\draw[fill = white] (1.5,1.5) circle (.25);
\draw[fill = white] (2.5,1.5) circle (.25);
\draw[fill = white] (.5,.5) circle (.25);
\draw[thick] (1.5,.2) -- (1.5,.8);
\draw[thick] (1.2,.5) -- (1.8,.5);
\draw[fill = white] (2.5,.5) circle (.25);
\end{scope}
\end{tikzpicture}
\end{equation}
\noindent
If we undo the blue crossing uncrossing pair, the diagram transforms into the following. Note that the location of the white stone which was part of the yellow pair has moved.
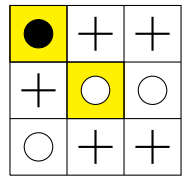
\begin{displaymath}
\begin{tikzpicture}
\begin{scope}[scale = .75]
\draw [fill=yellow] (0,3) rectangle (1,2);
\draw [fill=yellow] (1,2) rectangle (2,1);
\draw (0,0) -- (0,3) -- (3,3) -- (3,0) -- (0,0);
\draw (0,1) -- (3,1);
\draw (0,2) -- (3,2);
\draw (1,0) -- (1,3);
\draw (2,0) -- (2,3);
\draw[fill = black] (.5,2.5) circle (.25);
\draw[thick] (1.5,2.2) -- (1.5,2.8);
\draw[thick] (1.2,2.5) -- (1.8,2.5);
\draw[thick] (2.5,2.2) -- (2.5,2.8);
\draw[thick] (2.2,2.5) -- (2.8,2.5);
\draw[thick] (.5,1.2) -- (0.5,1.8);
\draw[thick] (.2,1.5) -- (.8,1.5);
\draw[fill = white] (1.5,1.5) circle (.25);
\draw[fill = white] (2.5,1.5) circle (.25);
\draw[fill = white] (.5,.5) circle (.25);
\draw[thick] (1.5,.2) -- (1.5,.8);
\draw[thick] (1.2,.5) -- (1.8,.5);
\draw[thick] (2.5,.2) -- (2.5,.8);
\draw[thick] (2.2,.5) -- (2.8,.5);
\end{scope}
\end{tikzpicture}
\end{displaymath}
\noindent
If we undo the yellow crossing/uncrossing pair, the diagram becomes the following. Note that in this case, the location of the blue black stone has moved.
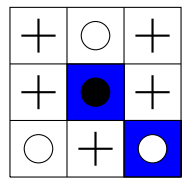
\begin{displaymath}
\begin{tikzpicture}
\begin{scope}[scale = .75]
\draw [fill=blue] (1,2) rectangle (2,1);
\draw [fill=blue] (2,1) rectangle (3,0);
\draw (0,0) -- (0,3) -- (3,3) -- (3,0) -- (0,0);
\draw (0,1) -- (3,1);
\draw (0,2) -- (3,2);
\draw (1,0) -- (1,3);
\draw (2,0) -- (2,3);
\draw[thick] (.5,2.2) -- (.5,2.8);
\draw[thick] (.2,2.5) -- (.8,2.5);
\draw[fill = white] (1.5,2.5) circle (.25);
\draw[thick] (2.5,2.2) -- (2.5,2.8);
\draw[thick] (2.2,2.5) -- (2.8,2.5);
\draw[thick] (.5,1.2) -- (0.5,1.8);
\draw[thick] (.2,1.5) -- (.8,1.5);
\draw[fill = black] (1.5,1.5) circle (.25);
\draw[thick] (2.5,1.2) -- (2.5,1.8);
\draw[thick] (2.2,1.5) -- (2.8,1.5);
\draw[fill = white] (.5,.5) circle (.25);
\draw[thick] (1.5,.2) -- (1.5,.8);
\draw[thick] (1.2,.5) -- (1.8,.5);
\draw[fill = white] (2.5,.5) circle (.25);
\end{scope}
\end{tikzpicture}
\end{displaymath}
\noindent
In either case, after undoing the last crossing/uncrossing pair and performing \Le-moves if necessary, we arrive at the following \Le-diagram.
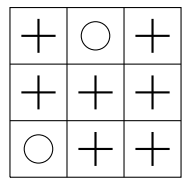
\begin{displaymath}
\begin{tikzpicture}
\begin{scope}[scale = .75]
\draw (0,0) -- (0,3) -- (3,3) -- (3,0) -- (0,0);
\draw (0,1) -- (3,1);
\draw (0,2) -- (3,2);
\draw (1,0) -- (1,3);
\draw (2,0) -- (2,3);
\draw[thick] (.5,2.2) -- (.5,2.8);
\draw[thick] (.2,2.5) -- (.8,2.5);
\draw[fill = white] (1.5,2.5) circle (.25);
\draw[thick] (2.5,2.2) -- (2.5,2.8);
\draw[thick] (2.2,2.5) -- (2.8,2.5);
\draw[thick] (.5,1.2) -- (0.5,1.8);
\draw[thick] (.2,1.5) -- (.8,1.5);
\draw[thick] (1.5,1.2) -- (1.5,1.8);
\draw[thick] (1.2,1.5) -- (1.8,1.5);
\draw[thick] (2.5,1.2) -- (2.5,1.8);
\draw[thick] (2.2,1.5) -- (2.8,1.5);
\draw[fill = white] (.5,.5) circle (.25);
\draw[thick] (1.5,.2) -- (1.5,.8);
\draw[thick] (1.2,.5) -- (1.8,.5);
\draw[thick] (2.5,.2) -- (2.5,.8);
\draw[thick] (2.2,.5) -- (2.8,.5);
\end{scope}
\end{tikzpicture}
\end{displaymath}
Simultaneously replacing the stones the blue and yellow squares in (\ref{eqn:shaded_box_example}) with pluses yields
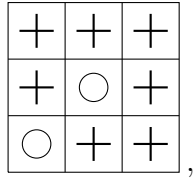
\begin{displaymath}
\begin{tikzpicture}
\begin{scope}[scale = .75]
\draw (0,0) -- (0,3) -- (3,3) -- (3,0) -- (0,0);
\draw (0,1) -- (3,1);
\draw (0,2) -- (3,2);
\draw (1,0) -- (1,3);
\draw (2,0) -- (2,3);
\draw[thick] (.5,2.2) -- (.5,2.8);
\draw[thick] (.2,2.5) -- (.8,2.5);
\draw[thick] (1.5,2.2) -- (1.5,2.8);
\draw[thick] (1.2,2.5) -- (1.8,2.5);
\draw[thick] (2.5,2.2) -- (2.5,2.8);
\draw[thick] (2.2,2.5) -- (2.8,2.5);
\draw[thick] (.5,1.2) -- (0.5,1.8);
\draw[thick] (.2,1.5) -- (.8,1.5);
\draw[fill = white] (1.5,1.5) circle (.25);
\draw[thick] (2.5,1.2) -- (2.5,1.8);
\draw[thick] (2.2,1.5) -- (2.8,1.5);
\draw[fill = white] (.5,.5) circle (.25);
\draw[thick] (1.5,.2) -- (1.5,.8);
\draw[thick] (1.2,.5) -- (1.8,.5);
\draw[thick] (2.5,.2) -- (2.5,.8);
\draw[thick] (2.2,.5) -- (2.8,.5);
\draw (3.2,0) node {$,$};
\end{scope}
\end{tikzpicture}
\end{displaymath}
\noindent
which is not a reduced diagram for the same permutation.
\end{example}

The problem of black stones moving around when undoing crossing/uncrossing pairs can be solved by undoing these crossing uncrossing pairs as black stones increase in the $\prec$ partial order. So, in our example we first undo the blue crossing/uncrossing pair, then undo the yellow one. The problem of the white stones involved in crossing/uncrossing pairs moving around is however unavoidable. The following algorithm provides an inductive procedure to compute the {\textit{partner}} of a box in a $\bullet/\circ/+$-diagram.

\begin{algorithm} \label{alg:partner}
Given a box $b = (i_b,j_b)$, suppose all boxes containing black stones in $b^{in}$ aside from $b$ have been assigned partners.
\begin{itemize}
\item[1.] If there is no black stone or plus to the right of $b$ in row $i_b$ or no black stone or plus below $b$ in column $j_b$, then $b$ has no partner.
\item[2.] Trace right from $b$ row $i_b$ until you hit a black stone or a plus in a box $c = (i_b,j_c)$ and down from $b$ down in column $j_b$ until you hit a black stone or plus in a box $d = (i_d,j_b)$.
\item[3.] If any of the following situations occur, $b$ has no partner.
\begin{itemize}
\item[3.1.] There is no box $e = (i_d,j_c)$.
\item[3.2.] There is a plus or a black stone in a square $(i,j_c)$ with $i_b < i < i_d$.
\item[3.3.] There is a plus or a black stone in a square $(i_d,j)$ with $j_c < j < j_b$.
\end{itemize}
\item[4.] Otherwise, if $e$ contains a white stone, $e$ is $b$'s partner. Let $P_b$ be the path from $b$ right to $c$ then down to $e$ and let $Q_b$ be the path from $b$ down to $d$ then right to $e$.
\item[5.] If $e$ contains a plus or a black stone, construct a path $P_b$ starting at $b$ traveling to the right via the following procedure:
\begin{itemize}
\item[5.1.] If $P_b$ hits a plus while traveling right, switch from traveling right to down;
\item[5.2.] If $P_b$ hits a plus while traveling down, switch from traveling down to right;
\item[5.3.] If $P_b$ hits a black stone while traveling right, jump to that black stone's partner and continue traveling down;
\item[5.4.] If $P_b$ hits a black stone while traveling down, jump to that black stone's partner and continue traveling right;
\item[5.5.] If $P_b$ hits a white stone that was partnered with some other black stone in $b^{in}$ while traveling right, switch from traveling right to down;
\item[5.6.] If $P_b$ hits a white stone that was partnered with some other black stone in $b^{in}$ while traveling down, switch from traveling down to right.
\end{itemize}
\item[6.] Construct a path $Q_b$ starting at $b$ and traveling down following the same rules.
\item[7.] If $P_b$ and $Q_b$ meet and the first square they meet in (the largest square they meet in the $\prec$ partial order) contains a white stone, that square is $b$'s partner.
\item[8.] Otherwise, $b$ has no partner.
\end{itemize}
\end{algorithm}

\begin{example} \label{ex:partners}
Consider the following diagram, which is a Go-diagram.
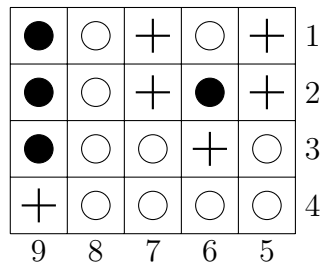
\begin{displaymath}
\begin{tikzpicture}
\begin{scope}[scale = .75]
\draw[step=1] (0,1) grid (5,5);
\draw (5.3,4.5) node {$1$};
\draw (5.3,3.5) node {$2$};
\draw (5.3,2.5) node {$3$};
\draw (5.3,1.5) node {$4$};
\draw (4.5,.7) node {$5$};
\draw (3.5,.7) node {$6$};
\draw (2.5,.7) node {$7$};
\draw (1.5,.7) node {$8$};
\draw (.5,.7) node {$9$};
%
%
\draw[thick] (.5,1.2) -- (.5,1.8);  \draw[thick] (.2,1.5) -- (.8,1.5);
\draw[fill = white] (1.5,1.5) circle (.25);
\draw[fill = white] (2.5,1.5) circle (.25);
\draw[fill = white] (3.5,1.5) circle (.25);
\draw[fill = white] (4.5,1.5) circle (.25);
%
\draw[fill = black] (.5,2.5) circle (.25);
\draw[fill = white] (1.5,2.5) circle (.25);
\draw[fill = white] (2.5,2.5) circle (.25);
\draw[thick] (3.5,2.2) -- (3.5,2.8);  \draw[thick] (3.2,2.5) -- (3.8,2.5);
\draw[fill = white] (4.5,2.5) circle (.25);
%
\draw[fill = black] (.5,3.5) circle (.25);
\draw[fill = white] (1.5,3.5) circle (.25);
\draw[thick] (2.5,3.2) -- (2.5,3.8);  \draw[thick] (2.2,3.5) -- (2.8,3.5);
\draw[fill = black] (3.5,3.5) circle (.25);
\draw[thick] (4.5,3.2) -- (4.5,3.8);  \draw[thick] (4.2,3.5) -- (4.8,3.5);
%
\draw[fill = black] (.5,4.5) circle (.25);
\draw[fill = white] (1.5,4.5) circle (.25);
\draw[thick] (2.5,4.2) -- (2.5,4.8);  \draw[thick] (2.2,4.5) -- (2.8,4.5);
\draw[fill = white] (3.5,4.5) circle (.25);
\draw[thick] (4.5,4.2) -- (4.5,4.8);  \draw[thick] (4.2,4.5) -- (4.8,4.5);
\end{scope}
\end{tikzpicture}
\end{displaymath}
None of the boxes in row $4$ and in columns $5$ and $8$ will have partners, since none of these boxes have a plus or a black stone both below them and to their right. The boxes $(2,7), (3,6),$ and $(3,7)$ do not have partners for similar reasons.

It is straight forward to see that the black stone in box $(2,6)$ is partnered with the white stone in box $(3,5)$; the black stone in box $(3,9)$ is partnered with the white stone in box $(4,6)$; and the black stone in box $(2,9)$ is partnered with the white stone in box $(3,7)$.

The box in $(1,7)$ has pluses to its right and below it. However, from these pluses, if we try to trace down from $(1,5)$ and right from $(2,7)$ to where they meet in $(2,5)$, we notice there is a black stone along the line from $(2,7)$ to $(2,5)$. Since the construction only allowed for white stones along these lines, $(1,7)$ does not have a partner.

The box $(2,6)$ has boxes with pluses or black stones below it and to its right. Tracing right from $(2,6)$ and down from $(1,5)$, everything is fine. Since there is a plus in $(2,5)$, we must construct paths $P$ and $Q$ as dictated by the construction.

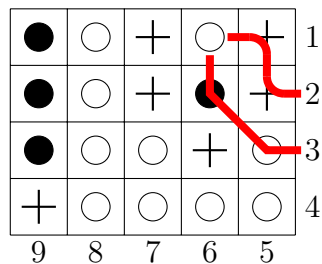
\begin{displaymath}
\begin{tikzpicture}
\begin{scope}[scale = .75]
\draw[step=1] (0,1) grid (5,5);
\draw (5.3,4.5) node {$1$};
\draw (5.3,3.5) node {$2$};
\draw (5.3,2.5) node {$3$};
\draw (5.3,1.5) node {$4$};
\draw (4.5,.7) node {$5$};
\draw (3.5,.7) node {$6$};
\draw (2.5,.7) node {$7$};
\draw (1.5,.7) node {$8$};
\draw (.5,.7) node {$9$};
%
\draw[thick] (.5,1.2) -- (.5,1.8);  \draw[thick] (.2,1.5) -- (.8,1.5);
\draw[fill = white] (1.5,1.5) circle (.25);
\draw[fill = white] (2.5,1.5) circle (.25);
\draw[fill = white] (3.5,1.5) circle (.25);
\draw[fill = white] (4.5,1.5) circle (.25);
%
\draw[fill = black] (.5,2.5) circle (.25);
\draw[fill = white] (1.5,2.5) circle (.25);
\draw[fill = white] (2.5,2.5) circle (.25);
\draw[thick] (3.5,2.2) -- (3.5,2.8);  \draw[thick] (3.2,2.5) -- (3.8,2.5);
\draw[fill = white] (4.5,2.5) circle (.25);
%
\draw[fill = black] (.5,3.5) circle (.25);
\draw[fill = white] (1.5,3.5) circle (.25);
\draw[thick] (2.5,3.2) -- (2.5,3.8);  \draw[thick] (2.2,3.5) -- (2.8,3.5);
\draw[fill = black] (3.5,3.5) circle (.25);
\draw[thick] (4.5,3.2) -- (4.5,3.8);  \draw[thick] (4.2,3.5) -- (4.8,3.5);
%
\draw[fill = black] (.5,4.5) circle (.25);
\draw[fill = white] (1.5,4.5) circle (.25);
\draw[thick] (2.5,4.2) -- (2.5,4.8);  \draw[thick] (2.2,4.5) -- (2.8,4.5);
\draw[fill = white] (3.5,4.5) circle (.25);
\draw[thick] (4.5,4.2) -- (4.5,4.8);  \draw[thick] (4.2,4.5) -- (4.8,4.5);
%
\draw[line width=3pt, red] (3.81,4.5) -- (4.2,4.5);
\draw[line width=3pt, red] (4.2,4.5) arc (90:0:.3cm);
\draw[line width=3pt, red] (4.5,4.2) -- (4.5,3.8);
\draw[line width=3pt, red] (4.8,3.5) arc (270:180:.3cm);
\draw[line width=3pt, red] (4.8,3.5) -- (5.1,3.5);
%
\draw[line width=3pt, red] (3.5,4.18) -- (3.5,3.5) -- (4.5,2.5) -- (5.1,2.5);
\end{scope}
\end{tikzpicture}
\end{displaymath}

\noindent
Since these paths do not meet, $(2,6)$ does not have a partner.

Finally, the box $(1,9)$ has boxes with pluses or black stones below it and to its right. Drawing out the paths $P$ and $Q$ as described, we obtain the following. Note that the path $Q$ takes a turn in box $(4,6)$ because that box was partnered with the black stone in $(3,9)$.

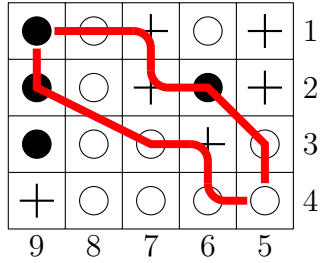
\begin{displaymath}
\begin{tikzpicture}
\begin{scope}[scale = .75]
\draw[step=1] (0,1) grid (5,5);
\draw (5.3,4.5) node {$1$};
\draw (5.3,3.5) node {$2$};
\draw (5.3,2.5) node {$3$};
\draw (5.3,1.5) node {$4$};
\draw (4.5,.7) node {$5$};
\draw (3.5,.7) node {$6$};
\draw (2.5,.7) node {$7$};
\draw (1.5,.7) node {$8$};
\draw (.5,.7) node {$9$};
%
%
\draw[thick] (.5,1.2) -- (.5,1.8);  \draw[thick] (.2,1.5) -- (.8,1.5);
\draw[fill = white] (1.5,1.5) circle (.25);
\draw[fill = white] (2.5,1.5) circle (.25);
\draw[fill = white] (3.5,1.5) circle (.25);
\draw[fill = white] (4.5,1.5) circle (.25);
%
\draw[fill = black] (.5,2.5) circle (.25);
\draw[fill = white] (1.5,2.5) circle (.25);
\draw[fill = white] (2.5,2.5) circle (.25);
\draw[thick] (3.5,2.2) -- (3.5,2.8);  \draw[thick] (3.2,2.5) -- (3.8,2.5);
\draw[fill = white] (4.5,2.5) circle (.25);
%
\draw[fill = black] (.5,3.5) circle (.25);
\draw[fill = white] (1.5,3.5) circle (.25);
\draw[thick] (2.5,3.2) -- (2.5,3.8);  \draw[thick] (2.2,3.5) -- (2.8,3.5);
\draw[fill = black] (3.5,3.5) circle (.25);
\draw[thick] (4.5,3.2) -- (4.5,3.8);  \draw[thick] (4.2,3.5) -- (4.8,3.5);
%
\draw[fill = black] (.5,4.5) circle (.25);
\draw[fill = white] (1.5,4.5) circle (.25);
\draw[thick] (2.5,4.2) -- (2.5,4.8);  \draw[thick] (2.2,4.5) -- (2.8,4.5);
\draw[fill = white] (3.5,4.5) circle (.25);
\draw[thick] (4.5,4.2) -- (4.5,4.8);  \draw[thick] (4.2,4.5) -- (4.8,4.5);
%
\draw[line width=3pt, red] (.81,4.5) -- (2.2,4.5);
\draw[line width=3pt, red] (2.2,4.5) arc (90:0:.3cm);
\draw[line width=3pt, red] (2.5,4.2) -- (2.5,3.8);
\draw[line width=3pt, red] (2.5,3.8) arc (180:270:.3cm);
\draw[line width=3pt, red] (2.8,3.5) -- (3.5,3.5) -- (4.5,2.5) -- (4.5,1.81);
%
\draw[line width=3pt, red] (.5,4.19) -- (.5,3.5) -- (2.5,2.5) -- (3.2,2.5);
\draw[line width=3pt, red] (3.2,2.5) arc (90:0:.3cm);
\draw[line width=3pt, red] (3.5,2.2) -- (3.5,1.8);
\draw[line width=3pt, red] (3.5,1.8) arc (180:270:.3cm);
\draw[line width=3pt, red] (3.8,1.5) -- (4.19,1.5);
\end{scope}
\end{tikzpicture}
\end{displaymath}
Since these two paths meet in a white stone at $(4,5)$, box $(1,9)$ is partnered to box $(4,5)$.
\end{example}

\begin{proposition} \label{prop:collecting_facts}
Suppose $b = (i_b,j_b)$ is a box in a $\bullet/\circ/+$-diagram $D$ which has a partner.
\begin{itemize}
\item[(i)] $b$'s partner is in $b^{in}$.
\item[(ii)] If the paths $P_b$ and $Q_b$ in Algorithm \ref{alg:partner} are constructed, all boxes encountered along these paths are in $b^{in}$.
\item[(iii)] $b$'s partner is a white stone.
\item[(iv)] No other box shares a partner with $b$.
\item[(v)] Let $c = (i_b,j_c)$, $d = (i_d,j_b)$, and $e = (i_d,j_c)$ be as in Algorithm \ref{alg:partner}. If there is a plus or a black stone in $e$, one of $c$ or $d$ must contain a black stone.
\item[(vi)] Let $a \in D$ be incomparable to $b$ in the $\prec$ partial order. Suppose $a$ has a partner $p$. Then $p$ is not along the paths $P_b$ or $Q_b$ from $b$ constructed in Algorithm \ref{alg:partner}.
\end{itemize}
\end{proposition}

\begin{proof}
Points (i), (ii), and (iii) are immediately apparent from Algorithm \ref{alg:partner}. To prove point (iv), note that Algorithm \ref{alg:partner} is reversible.

For (v), suppose the box $e$ contains a black stone. Then, we must construct the paths $P_b$ and $Q_b$ in Algorithm \ref{alg:partner}. If both $c$ and $d$ contain pluses, then these paths first meet at the box $e$. So, $e$ is the partner of $b$, which contradicts point (iii) in this proposition. We remark that the case of point (v) where $e$ contains a plus is an artifact of the distinguished property for subwords. The case where $e$ contains a black stone is an artifact of the fact the crossings and uncrossings must alternate. 

For (vi), suppose the path $P_b$ goes through $p$. As remarked in the proof of (iii), the construction of these paths is reversible. So, $P_b$ agrees with $P_a$ or $Q_a$ eventually and thus goes through $a$ eventually, either before or after $b$. But, everything along paths from $b$ is in $b^{in}$ and everything along paths from $a$ is in $a^{in}$ by point (ii). This contradicts the incomparability of $a$ and $b$. \qedhere
\end{proof}

Observe that in Example \ref{ex:partners} the partners $(2,9)$ and $(3,7)$ do not constitute a crossing/uncrossing pair in the pipe dream associated to the diagram. However, if we change the boxes $(3,9)$ and $(4,6)$ to pluses, the black stone in $(2,9)$ and the white stone in $(3,8)$ will be a crossing/noncrossing pair. In fact, for any black stone in a box $b$ in this diagram, if we flip all of the black stones and their partners in $b^{in}$ to pluses, $b$ and its partner form a crossing/uncrossing pair in the new diagram. This observation generalizes.

Consider a $\bullet/\circ/+$-diagram $D$ such a square has a partner if and only if it is filled with a black stone. We'll see shortly that such diagrams are exactly Go-diagrams. Let $b$ be a box in a diagram $D$. Let $f(D,b)$ be the diagram obtained by replacing all black stones in boxes $c \prec b$ and all white stones in the partners of these boxes with pluses. Consider the following pair of properties:

\begin{itemize}
\item[(P1)] $v^{f(D,b)}_{b^{in}} = v^{D}_{b^{in}}$.
\item[(P2)] $b$ has a partner $p$ if and only if the boxes $b$ and $p$ form a crossing/uncrossing pair in $f(D,b)$.
\end{itemize}

\begin{lemma} \label{lem:inductive_step}
Let $D$ be a $\bullet/\circ/+$-diagram such that a square has a partner if and only if it is filled with a black stone. Let $b$ be a box in $D$ such that properties (i) and (ii) hold for all $c \prec b$. Then, properties (P1) and (P2) hold for $b$.
\end{lemma}

\begin{proof}
Let $b$ be a box in $D$ and suppose that properties (P1) and (P2) hold for all boxes in $b^{in}$ aside from $b$. Let $c$ be the box directly to the right of $b$ if such a box exists. We may flip all of the the black stones in $c^{in}$ and their partners to pluses to obtain a diagram $D'$ without changing the permutation $v_{c^{in}}$. At this point, we have flipped all black stones and their partners in $b^{in}$ aside from those in $b$'s column. We proceed to flip the black stones in this column and their partners starting from the bottom of the column.

Let $d$ be the lowest box containing a black stone in the same column as $b$. From point (vi) in Proposition \ref{prop:collecting_facts}, the only squares along the paths $P_d$ and $Q_d$ from $d$ to its partner which were flipped in passing to $D'$ are black stones in $d^{in}$ and partners of these stones. In the pipe dream of $f(D,d)$, follow the pipes coming out of the box $d$ down and to the right. The pipe going to the right turns downward at the first plus it encounters; such a plus could have come from either a plus or a black in the original diagram $D$. Likewise, the pipe going down from $d$ turns right at the first plus it encountered. Point 5.2 in Algorithm  \ref{alg:partner} and point (vi) in Proposition \ref{prop:collecting_facts} guarantee that, after these initial turns these pipes continue without turning until they meet at some square $e$. If $e$ contained a white stone in $D$, it still contains a white stone in $f(D,d)$. In this case, $d$ and $e$ were partnered in $D$ and they form a crossing/uncrossing pair in $f(D,d)$.

If the square $e$ contained a plus or black stone in $D$, it will contain a plus in $f(D,d)$. So, the pipes originating at $d$ will continue to travel down and right according to the rules:
\begin{itemize}
\item[1.] If they hit a plus that was a plus in $D$ while traveling right, switch from traveling right to traveling down.
\item[2.] If they hit a plus that was a plus in $D$ while traveling down, switch from traveling right to traveling right.
\item[3.] If they hit a plus that was a black stone in $D$, our inductive assumption tells us this plus was a crossing uncrossing pair with its partner. So, if they hit a plus that was a black stone in $D$ while traveling right, continue to its partner, then switch to traveling downward.
\item[4.] If they hit a plus that was a black stone in $D$ while traveling down, continue to its partner, then switch to traveling to the right.
\item[5.] If they hit a plus that was a white stone in $D$ while traveling right, switch from traveling right to traveling down. Proposition \ref{prop:collecting_facts} point (vi) guarantees such a white stone in $D$ had to be the partner of some square in $d^{in}$.
\item[6.] If they hit a plus that was a white stone in $D$ while traveling right, switch from traveling right to traveling down.
\end{itemize}
\noindent 
This list of rules agrees with points 5.1--5.6 in Algorithm \ref{alg:partner}. So, the pipes originating at $d$ next share a square at the same point that the paths $P_d$ and $Q_d$ from Algorithm \ref{alg:partner} meet. This square contains a white stone and is $d$'s partner. So, $d$ and $d$'s partner from $D$ form a crossing/uncrossing pair in $f(D,d)$. Then, flipping $d$ and $d$'s partner both to be pluses leaves the permutation unchanged.

Continuing in this way, we may flip all of the black stones in the same column as $b$ and their partners to pluses without altering the permutation. So, $v^{f(D,b)}_{b^{in}} = v^{D}_{b^{in}}$. In the case where $b$ contains a black stone, the same agrument as above shows that $b$ and its partner from $D$ form a crossing/uncrossing pair in $f(D,b)$.
\end{proof}

\begin{theorem} \label{thm:characterization}
A $\bullet/\circ/+$-diagram $D$ is a Go-diagram if and only if all boxes containing black stones have partners and all boxes with partners are filled with black stones. Changing all black stones and their partners to pluses simultaneously yields a reduced diagram for the pair of permutations determined by $D$.
\end{theorem}

\begin{proof}
Let $D$ be a Go-diagram and let $b$ be a box containing a black stone such that no box in $b^{in}$ contains a black stone. Since there are no black stones in $b^{in}$, property (P1) holds for $b$. Consider the diagram obtained by restricting $D$ to the subdiagram $b^{in} \cup b$ and replacing the black stone in $b$ with a plus. This diagram contains no black stones, and it is not a \Le-diagram. So, it contains some subdiagram violating the \Le-condition, of the form (\ref{eqn:le-violation}). Necessarily, $b$ is the top left corner of this subdiagram. Then, the bottom right corner of this subdiagram is $b$'s partner in $D$. Evidently, these two squares also form a crossing/uncrossing pair. So, property (P2) holds for $b$. Then, applying Lemma \ref{lem:inductive_step} inductively, properties (P1) and (P2) hold for all squares in $D$.

Now, let $b$ be any square in $b$. We want to show that $b$ contains a black stone if and only if it has a partner in $D$. From the distinguished property, $b$ contains a black stone if and only if
\begin{equation} \label{eqn:distinguished}
\ell\left(v^{D}_{b^{in}} s_b\right) < \ell\left(v^{D}_{b^{in}}\right).
\end{equation}
Then, (P1) says that (\ref{eqn:distinguished}) holds if and only if
\begin{displaymath}
\ell\left(v^{f(D,b)}_{b^{in}} s_b\right) < \ell\left(v^{f(D,b)}_{b^{in}}\right).
\end{displaymath}
This inequality holds if and only if $b$ forms a crossing/uncrossing pair with some box $p$ in $f(D,b)$. Property (P2) says that $b$ forms a crossing/uncrossing pair with $p$ if and only if $p$ is $b$'s partner in $D$.

Now, suppose $D$ is a diagram such that a square contains a black stone if and only if it has a partner. Let $b$ be a box in $D$ such that no boxes in $b^{in}$ contain black stones. Then, from point (v) in Proposition \ref{prop:collecting_facts}, $b$'s partner is defined by a diagram of the following form.
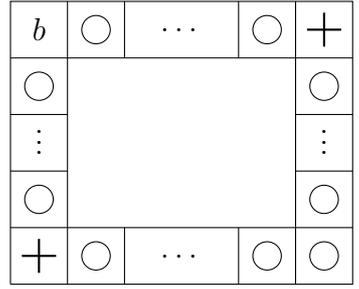
\begin{displaymath}
\begin{tikzpicture}
\begin{scope}[scale=.75]
\draw (0,0) -- (0,5) -- (6,5) -- (6,0) -- (0,0);
\draw (1,0) -- (1,5);
\draw (0,4) -- (6,4);
\draw (5,5) -- (5,0);
\draw (6,1) -- (0,1);
\draw (0,2) -- (1,2);
\draw (0,3) -- (1,3);
\draw (2,5) -- (2,4);
\draw (4,5) -- (4,4);
\draw (5,3) -- (6,3);
\draw (5,2) -- (6,2);
\draw (4,0) -- (4,1);
\draw (2,0) -- (2,1);

\draw[thick] (.5,.2) -- (.5,.8);
\draw[thick] (.2,.5) -- (.8,.5);
%
\draw (.5,4.5) node {$b$};
%
\draw[thick] (5.5,4.2) -- (5.5,4.8);
\draw[thick] (5.2,4.5) -- (5.8,4.5);
%
\draw[fill = white] (5.5,.5) circle (.25);
%
\draw[fill = white] (.5,1.5) circle (.25);
\draw[fill = white] (.5,3.5) circle (.25);
\draw[fill = white] (1.5,.5) circle (.25);
\draw[fill = white] (1.5,4.5) circle (.25);
\draw[fill = white] (4.5,4.5) circle (.25);
\draw[fill = white] (4.5,.5) circle (.25);
\draw[fill = white] (5.5,1.5) circle (.25);
\draw[fill = white] (5.5,3.5) circle (.25);
%
\draw (3,.5) node {$\dots$};
\draw (3,4.5) node {$\dots$};
\draw (.5,2.65) node {$\vdots$};
\draw (5.5,2.65) node {$\vdots$};
\draw (6.2,0) node {$,$};
\end{scope}
\end{tikzpicture}
\end{displaymath}
\noindent
where the interior of the diagram could be filled with anything. So, $b$ and its partner form a crossing/uncrossing pair. Then, applying Lemma \ref{lem:inductive_step} inductively, properties (P1) and (P2) hold for all squares in $D$.

Now, let $b$ be any box in $D$. To verify that $D$ is a Go-diagram we must check that $b$ contains a black stone if and only if (\ref{eqn:distinguished}) holds for $b$. Since (P1) holds for every box in $D$, $b$ satisfies (\ref{eqn:distinguished}) if and only if 
\begin{displaymath}
\ell\left(v^{f(D,b)}_{b^{in}} s_b\right) < \ell\left(v^{f(D,b)}_{b^{in}}\right).
\end{displaymath}
\noindent
This condition holds if and only if $b$ forms a crossing/uncrossing pair with some box $p$ in $f(D,b)$. Then, Property (P2) says that $b$ forms a crossing/uncrossing pair with $p$ in $f(D,b)$ if and only if $p$ is $b$'s partner in $D$. From our assumption, $b$ has a partner in $D$ if and only $b$ contains a black stone.
\end{proof}

Theorem \ref{thm:characterization} may be used to give intriguing partial lists of forbidden subdiagrams for the class of Go-diagrams. 
Though Theorem \ref{thm:sad_injection} demonstrates that there is no finite characterization of Go-diagrams in terms of forbidden subdiagrams, such tests can still be valuable as a quick reality check for whether a diagram is or is not a Go-diagram.

\begin{corollary} \label{cor:forbidden_go_diagram}
Any Go-diagram avoids subdiagrams of the form
%

\begin{equation} \label{eqn:forbidden_go_diagram}
\begin{tikzpicture}
\begin{scope}[scale=.75]
\draw (0,0) -- (0,4) -- (6,4) -- (6,0) -- (0,0);
\draw (1,0) -- (1,1) -- (0,1);
\draw (0,3) -- (1,3) -- (1,4);
\draw (5,4) -- (5,3) -- (6,3);
\draw (6,1) -- (5,1) -- (5,0);
%
%
\draw (0,0) -- (1,1);
\draw[fill = black] (.3,.7) circle (.18);
\draw[thick] (.7,.1) -- (.7,.5);
\draw[thick] (.5,.3) -- (.9,.3);
%
\begin{scope}[shift={(0,3)}]
\draw (0,0) -- (1,1);
\draw[fill = white] (.3,.7) circle (.18);
\draw[thick] (.7,.1) -- (.7,.5);
\draw[thick] (.5,.3) -- (.9,.3);
\end{scope}
%
\begin{scope}[shift={(5,3)}]
\draw (0,0) -- (1,1);
\draw[fill = black] (.3,.7) circle (.18);
\draw[thick] (.7,.1) -- (.7,.5);
\draw[thick] (.5,.3) -- (.9,.3);
\end{scope}
%
\draw[fill = white] (5.5,.5) circle (.25);
\draw (3,2) node {white stones};
\draw (6.2,0) node {$,$};
\end{scope}
\end{tikzpicture}
\end{equation}
\noindent
where the boxes with slashes in them indicate that the box could be filled with the items on either side of the slash.
\end{corollary}

\begin{proof}
From Theorem \ref{thm:characterization} and Point 4 in Algorithm \ref{alg:partner}, it is obvious a Go diagram must avoid subdiagrams of the form

\begin{equation} \label{eqn:obvious_forbidden_diagram}
\begin{tikzpicture}
\begin{scope}[scale=.75]
\draw (0,0) -- (0,5) -- (6,5) -- (6,0) -- (0,0);
\draw (1,0) -- (1,5);
\draw (0,4) -- (6,4);
\draw (5,5) -- (5,0);
\draw (6,1) -- (0,1);
\draw (0,2) -- (1,2);
\draw (0,3) -- (1,3);
\draw (2,5) -- (2,4);
\draw (4,5) -- (4,4);
\draw (5,3) -- (6,3);
\draw (5,2) -- (6,2);
\draw (4,0) -- (4,1);
\draw (2,0) -- (2,1);

\draw (0,0) -- (1,1);
\draw[fill = black] (.3,.7) circle (.18);
\draw[thick] (.7,.1) -- (.7,.5);
\draw[thick] (.5,.3) -- (.9,.3);
%
\begin{scope}[shift={(0,4)}]
\draw (0,0) -- (1,1);
\draw[fill = white] (.3,.7) circle (.18);
\draw[thick] (.7,.1) -- (.7,.5);
\draw[thick] (.5,.3) -- (.9,.3);
\end{scope}
%
\begin{scope}[shift={(5,4)}]
\draw (0,0) -- (1,1);
\draw[fill = black] (.3,.7) circle (.18);
\draw[thick] (.7,.1) -- (.7,.5);
\draw[thick] (.5,.3) -- (.9,.3);
\end{scope}
%
\draw[fill = white] (5.5,.5) circle (.25);
%
\draw[fill = white] (.5,1.5) circle (.25);
\draw[fill = white] (.5,3.5) circle (.25);
\draw[fill = white] (1.5,.5) circle (.25);
\draw[fill = white] (1.5,4.5) circle (.25);
\draw[fill = white] (4.5,4.5) circle (.25);
\draw[fill = white] (4.5,.5) circle (.25);
\draw[fill = white] (5.5,1.5) circle (.25);
\draw[fill = white] (5.5,3.5) circle (.25);
%
\draw (3,.5) node {$\dots$};
\draw (3,4.5) node {$\dots$};
\draw (.5,2.65) node {$\vdots$};
\draw (5.5,2.65) node {$\vdots$};
\draw (6.2,0) node {$,$};
\end{scope}
\end{tikzpicture}
\end{equation}
\noindent
where it is not specified what the interior of the diagram is filled with. We show the existence of such a subdiagram $D$ implies the existence of a diagram of the form (\ref{eqn:forbidden_go_diagram}). Suppose there is a plus or a black stone in the interior of this diagram. Choose a plus or black stone in a box $b$ in the interior of $D$ such that there are no interior pluses or black stones to the right of $b$. Now, let $c$ be the highest box in the same column as $b$ containing a black stone or plus. Then, the square whose bottom left corner is $c$ and whose top right corner is the top right corner is the top right corner of $D$ is a subdiagram of the form (\ref{eqn:forbidden_go_diagram}).
\end{proof}

\end{document}